\documentclass[11pt, a4paper,reqno]{amsart}
\usepackage[a4paper, margin=4cm]{}
\usepackage{amssymb,amsfonts, amsthm, xfrac, amscd}

\usepackage{color}
\usepackage{array}

\usepackage{color}
\usepackage[bookmarks]{hyperref}

\allowdisplaybreaks[4]

\newtheoremstyle{myremark}     {10pt}{10pt}{}{}{\bfseries}{.}{.5em}{}

\newtheorem{thm}{Theorem}[section]
\newtheorem{cor}[thm]{Corollary}
\newtheorem{lem}[thm]{Lemma}
\newtheorem{pro}[thm]{Proposition}
\theoremstyle{definition}
\newtheorem{defn}[thm]{Definition}

\theoremstyle{myremark}
\newtheorem{rem}[thm]{Remark}
  
\begin{document}

\title{Stein's Method for Tempered Stable Distributions}

\author[Barman]{Kalyan Barman}
\address{\hskip-\parindent
Kalyan Barman, Department of Mathematics, IIT Madras,
Chennai - 600036, India.}

\email{barmankalyan.iitm@gmail.com}

\author[Upadhye]{Neelesh S Upadhye}

\address{\hskip-\parindent
Neelesh S Upadhye, Department of Mathematics, IIT Madras,
Chennai - 600036, India.}
\email{neelesh@iitm.ac.in}



\begin{abstract}
\noindent
In this article, we develop Stein characterization for two-sided tempered stable distribution. Stein characterizations for normal, gamma, Laplace, and variance-gamma distributions already known in the literature follow easily. One can also derive Stein characterizations for more difficult distributions such as the distribution of product of two normal random variables, a difference between two gamma random variables. Using the semigroup approach, we obtain estimates of the solution to Stein equation. Finally, we apply these estimates to obtain error bounds in the Wasserstein-type distance for tempered stable approximation in three well-known problems: comparison between two tempered stable distributions, Laplace approximation of random geometric sums, and six moment theorem for the symmetric variance-gamma approximation of functionals of double Wiener-It$\ddot{\text{o}}$ integrals. We also compare our results with the existing literature.  
\end{abstract}

\maketitle

\section{Introduction}\label{Sec:Introduction}

\noindent
Stein's method introduced by Charles Stein \cite{k2} is a powerful approach for deriving bounds for normal approximation. 
The method is based on the simple fact that, any real-valued random variable $Z$ has $\mathcal{N}(0,1)$ distribution, if and only if 
\begin{equation*}
\mathbb{E}\left(f^{\prime}(Z)-Zf(Z) \right)=0,
\end{equation*}

\noindent
where $f$ is any real-valued absolutely continuous function such that $\mathbb{E}|f^{\prime}(Z)|<\infty$. This characterization leads us to the Stein equation
\begin{equation}\label{normal2}
f^{\prime}(x)-xf(x)=h(x)-\mathbb{E}h(Z),
\end{equation}
\noindent
where $h$ is a real-valued test function. Replacing $x$ with a random variable $Y$ and taking expectations on both sides of  (\ref{normal2}) gives
\begin{equation}\label{normal3}
\mathbb{E}\left(f^{\prime}(Y)-Yf(Y) \right)=\mathbb{E}h(Y)-\mathbb{E}h(Z).
\end{equation}

\noindent
This equality (\ref{normal3}) plays a crucial role in Stein's method. The $\mathcal{N}(0,1)$ distribution is characterized by (\ref{normal2}) such that the problem of bounding the quantity $|\mathbb{E}h(Y)-\mathbb{E}h(Z)|$ depends on smoothness of the solution to (\ref{normal2}) (see Section 2.2 of \cite{chen}), and behavior of $Y$. For more details on Stein's method, we refer to the reader the monograph \cite{k25}.

\noindent
Over the years, Stein's method has become one of the most popular tool for deriving bounds on the distance between two distributions and approximations to other classical distributions (see, \cite{k3,k4,k6,k5}). Stein's method for various families of distributions is also a topic of keen interest for researchers (see, for example, Pearson \cite{k23}, variance-gamma \cite{N2,k24,kk2}, discrete Gibbs measure \cite{k18,k19} family).

\noindent	
Recently, Arras and Houdr\'e \cite{k0,k22}, Chen et. al. \cite{k16,k17}, Upadhye and Barman \cite{k1}, Xu \cite{k20} have developed Stein's method for stable distributions. It is clear from the above articles that the derivation of Stein's method for the family of stable distributions is not straightforward due to the lack of symmetry and heavy-tailed behavior of stable distributions. One of the major obstacles in developing the method is the moments of stable distribution do not exist whenever the stability parameter $\alpha \in (0,1]$. To overcome these issues, different approaches and various assumptions are used to derive Stein's method for the family of stable distributions.\vspace{2mm}

\noindent
Tempered stable distributions (TSD) were first introduced by Koponen \cite{koponen} by tempering the tail properties of the stable distributions. TSD has mean, variance, exponential moments, and each TSD converges weakly to the stable distribution, whenever the tempering parameters tend to zero. For more details on TSD, we refer to the reader \cite{kk5}. Therefore TSD is an interesting family of probability distributions for researchers in probability theory as well as financial mathematics, see \cite{k14,k15,k9,k10}.\vspace{2mm}


\noindent
 K$\ddot{\text{u}}$chlar and Tappe \cite{k13} define two-sided and one-sided TSD as a six-parameter and three-parameter family of probability distributions, respectively. Again, TSD include many sub-families of distributions, such as CGMY, KoBol, bilateral-gamma, also the variance-gamma distributions, and as the special or limiting cases, the normal, gamma, Laplace, product of two normal and difference of two gamma distributions. Researchers in probability theory have widely studied the Stein's method for normal  \cite{k2}, gamma \cite{k5}, Laplace \cite{pike}, product-normal \cite{kk3} and variance-gamma \cite{k24,kk2} distributions. Therefore, it is of interest to develop the Stein's method for TSD and see its relation for the distributions mentioned above.\vspace{2mm}

\noindent
In this article, we obtain a Stein characterization for two-sided TSD using the characteristic function (cf) approach. It enables us to give the Stein characterizations for normal, gamma, Laplace, product of two normal, difference of two gamma, and variance-gamma distributions from the existing literature. Further, it also enables us to give new Stein characterizations for truncated L\'evy flight, CGMY, KoBol, and bilateral-gamma distributions. Next, we prove the existence of an additive size bias distribution for the one-sided case of TSD, in particular, the gamma distribution. Using the semigroup approach, we solve our Stein equation. We also derive some interesting estimates of the solution to Stein equation. Finally, we apply our estimates to obtain error bounds in the Wasserstein-type distance for tempered stable approximation in three well-known problems: comparison between two TSD, Laplace approximation of random geometric sums, and six moment theorem for the symmetric variance-gamma approximation of functionals of double Wiener-It$\ddot{\text{o}}$ integrals. We also compare our results with the existing literature.  \vspace{2mm}

\noindent
The organization of this article is as follows. Section \ref{pre} introduces some notations and preliminaries. In Section \ref{mr}, we state our results and their relevance to the existing literature on Stein characterization for TSD, in particular, for a sub-family of TSD, namely, the variance-gamma distributions (VGD). We solve our Stein equation by the semigroup approach. We also find estimates of the solution to Stein equation. In Section \ref{PP2:application}, we discuss three applications of our results.

\section{ Notations and Preliminaries}\label{pre}

\noindent
In this section, we review some preliminaries and known results used to develop Stein's method for TSD. Let us first discuss the large family of distributions, namely TSD.

\subsection{Tempered stable distributions} We first define the TSD and its related properties.
\begin{defn}(\cite[p.2]{k13})\label{PP2:TSD} A random variable $X$ having cf 
\begin{equation}\label{e1}
\phi(z)=\exp\left(\int_{\mathbb{R}}(e^{izu}-1)\nu(du)  \right),~~z\in\mathbb{R},
\end{equation}
and the L\'evy measure
\begin{equation}\label{e2}
\nu(du)=\left(\frac{\alpha^{+}  }{u^{1+\beta^{+}}}e^{-\lambda^{+}u}\mathbf{1}_{(0,\infty)}(u)+\frac{\alpha^{-}  }{|u|^{1+\beta^{-}}}e^{-\lambda^{-}|u|}\mathbf{1}_{(-\infty,0)}(u)\right)du
\end{equation}
is said to follow two-sided TSD with parameters $\alpha^{+},\lambda^{+},\alpha^{-},\lambda^{-}\in(0,\infty)$, and $\beta^{+},\beta^{-}\in[0,1)$, and it is denoted by $X\sim \text{TSD} (\alpha^{+},\beta^{+},\lambda^{+};\alpha^{-},\beta^{-},\lambda^{-})$.
 \end{defn}

\begin{defn}(\cite[p.2]{k13})\label{PP2:TSD0}
A random variable $X$ having cf \eqref{e1} and the L\'evy measure $	\nu(du)=\frac{\alpha^{+}  }{u^{1+\beta^{+}}}e^{-\lambda^{+} u}\mathbf{1}_{(0,\infty)}(u)du$ is said to follow one-sided TSD with positive support and parameters $\alpha^{+},\lambda^{+}\in (0,\infty)$ and $\beta^{+}\in [0,1)$, and it is denoted by $X\sim \text{TSD}_1(\alpha^{+},\beta^{+}, \lambda^{+})$.
\end{defn}
\noindent
Observe that, for $\alpha^{-}\to0^{+}$, the Definition \ref{PP2:TSD} reduces to Definition \ref{PP2:TSD0} which is the limiting distribution of $X\sim\text{TSD} (\alpha^{+},\beta^{+},\lambda^{+};\alpha^{-},\beta^{-},\lambda^{-}).$

\begin{defn}(\cite[p.2]{k13})\label{PP2:TSD1}
	A random variable $X$ having cf \eqref{e1} and the L\'evy measure $	\nu(du)=\frac{\alpha^{-}  }{u^{1+\beta^{-}}}e^{-\lambda^{-} u}\mathbf{1}_{(-\infty,0)}(u)du$ is said to follow one-sided TSD with negative support and parameters $\alpha^{+},\lambda^{+}\in (0,\infty)$ and $\beta^{+}\in [0,1)$, and it is denoted by $X\sim \text{TSD}_2(\alpha^{-},\beta^{-}, \lambda^{-})$.
\end{defn}
\noindent
Observe that, for $\alpha^{+}\to0^{+}$, the Definition \ref{PP2:TSD} reduces to Definition \ref{PP2:TSD1} which is the limiting distribution of $X\sim\text{TSD} (\alpha^{+},\beta^{+},\lambda^{+};\alpha^{-},\beta^{-},\lambda^{-}).$



	
%
	

%
\noindent
In the following remark, we note some important properties of TSD.

\begin{rem}\label{p2}
	\begin{enumerate}
		\item[(i)] With an appropriate choice of parameters, TSD cover truncated L\'evy flight, CGMY, KoBol, variance-gamma, bilateral-gamma distributions and others in the existing literature (see, \cite{k13} for more details). 
		
		\item [(ii)]	Let $X\sim \text{TSD}_1(\alpha^{+},0, \lambda^{+})$ with parameters $\alpha^{+},\lambda^{+} \in (0,\infty)$. Then, $X\sim Gamma(\alpha^{+},\lambda^{+})$.
		
		\item [(iii)] Let $X\sim \text{TSD}_2(\alpha^{-},0, \lambda^{-})$ with parameters $\alpha^{-},\lambda^{-} \in (0,\infty)$. Then, $-X\sim Gamma(\alpha^{-},\lambda^{-})$.
		

		\item [(iv)] It is known that density function of stable distributions can not be written in closed form for each $\alpha\in(0,1)$, where $\alpha$ is the stability parameter (see, p.33, \cite{newref1}). Indeed, non-Gaussian stable distributions have heavy tails, and they are asymptotically equivalent to Pareto distribution (see, \cite{newref2}).
		\item[(v)] TSD are designed by tempering the tail properties of the stable distributions (see, Remark 2.3, \cite{k13}). However, the density function of TSD may not be available in closed form but, K$\ddot{u}$chler and Tappe [Section 7, \cite{k13}]	have shown the existence of density function for each TSD with ``nice" asymptotic properties (see, \cite[Proposition 7.2 and Theorem 7.7]{k13}).

		
	\end{enumerate}
\end{rem}

\subsection{Variance-gamma distributions}
\noindent
Next, we discuss an important subclass of TSD, namely VGD. Let us define the various characterizations for VGD in terms of its cf. 

\begin{defn}\label{vgddf1}(K$\ddot{u}$chlar and Tappe \cite{k13})
	 A random variable $X$ with cf given by (\ref{e1}), and the L\'evy measure
	$$\nu_{VGD}(du)=\left(\frac{\alpha  }{u}e^{-\lambda^{+}u}\mathbf{1}_{(0,\infty)}(u)+\frac{\alpha  }{|u|}e^{-\lambda^{-}|u|}\mathbf{1}_{(-\infty,0)}(u)\right)du
	$$is said to follow a VGD with parameters $\alpha,\lambda^{+},\lambda^{-}\in (0,\infty)$, and it is denoted by $X \sim\text{VGD}_0(\alpha,\lambda^{+},$
	$\lambda^{-}).$ 	
\end{defn}

\noindent
Note here that,  $\text{VGD}_0(\alpha,\lambda^{+},\lambda^{-})\overset{d}{=}\text{TSD}(\alpha,0,\lambda^{+},\alpha,0,\lambda^{-})$, where $\overset{d}{=}$ denotes equality in distribution.
\begin{defn}\label{vgddf2}(Finlay and Seneta \cite{kk1}) A random variable $X$ having cf
	\begin{equation}\label{eqvg}
	\phi_{VGD_{1}}(z)=\left(1-iz\left(\frac{1}{\lambda^{+}}-\frac{1}{\lambda^{-}}\right)+\frac{z^{2}}{\lambda^{+}\lambda^{-}}  \right)^{-\alpha}, ~~z\in\mathbb{R}
	\end{equation}
	is said to follow a VGD with parameters $\alpha,\lambda^{+},\lambda^{-}\in(0,\infty)$, and it is denoted by $X \sim \text{VGD}_1(\alpha,$
	$\lambda^{+},\lambda^{-})$.
\end{defn}

\begin{defn}\label{vgddf3}
	 A random variable $X$ having cf  
	\begin{equation}\label{cfvg2}
	\phi_{VGD_{2}}(z)=\left(1-i2\theta z+\sigma^{2}z^{2}\right)^{-\frac{r}{2}},~~z\in\mathbb{R}
	\end{equation}
	\noindent
	is said to follow a VGD with parameters  $\sigma^{2},r \in (0,\infty)$ and $ \theta\in\mathbb{R},$  and it is denoted by $X\sim \text{VGD}_{2}(\sigma^{2},r,\theta)$.
	
\end{defn}


\noindent
Definition \ref{vgddf3} is used later for obtaining a Stein identity for VGD. In the following remark, we discuss relationship of the above representations with each other.

\begin{rem}
	Note that, the cf representations (\ref{eqvg}) and (\ref{e1}) (for $\nu_{VGD}$ defined in Definition \ref{vgddf1}) are exactly same by suitably adjusting the parameters, and by using Frullani's improper integral \cite{frullani} formula. For more details about this integral, we refer the reader to Appendix \ref{app}. Again, substituting $\frac{1}{\lambda^{+}\lambda^{-}}=\sigma^{2}, \left(\frac{1}{\lambda^{+}}-\frac{1}{\lambda^{-}}\right)=2\theta,$ and $\alpha=\frac{r}{2}$ in (\ref{eqvg}), we get (\ref{cfvg2}).
\end{rem}

\noindent
Next, we list the special and limiting cases of VGD (see, \cite{k24} for more details).


\begin{enumerate}
	\item [(O1)] Let $\sigma^{2}>0$ and a random variable $X_r$ has distribution $ \text{VGD}_{2}(\frac{\sigma^{2}}{r},r,0)$ with cf (\ref{cfvg2}). Then, $X_r$ weakly converges to $N(0,\sigma^{2})$, whenever $r\to \infty$. 
	
	\item [(O2)] Let $\alpha,\lambda>0$ and a random variable $X_\sigma$ has distribution $ \text{VGD}_{2}(\sigma^{2},2\alpha,(2\lambda)^{-1})$ with cf (\ref{cfvg2}). Then, $X_\sigma$ weakly converges to $Gamma(\alpha,\lambda)$, whenever $\sigma\to 0$.
	
	\item [(O3)] Let $X\sim N(0,\sigma_X^{2})$ and $Y\sim N(0,\sigma_Y^{2})$ are two independent normal random variables. Then, $XY \sim \text{VGD}_2(\sigma_X^{2}\sigma_Y^{2},1,0)$.
	
	\item [(O4)] Let $\sigma^{2}>0$, then the distribution of $\text{VGD}_2(\sigma^{2},2,0)$ has $Laplace(0,\sigma^{2})$ distribution. 
\end{enumerate}

\subsection{Function spaces and probability metrics}\label{PP2:FS}
\noindent
Next, we define a function space and suitable probability metric required to develop Stein's method for TSD. Let $\mathcal{S}(\mathbb{R})$ be the Schwartz space defined by
$$\mathcal{S}(\mathbb{R}):=\left\{f\in C^\infty(\mathbb{R}): \lim_{|x|\rightarrow \infty} |x^m\frac{d^{n}}{dx^{n}}f(x)|=0, \text{ for all } m,n\in \mathbb{N}\right\},$$
\noindent
where $C^\infty(\mathbb{R})$ is the class of infinitely differentiable functions on $\mathbb{R}$. It is important to note that the Fourier transform on $\mathcal{S}(\mathbb{R})$ is automorphism onto itself. This enables us to identify the elements of dual space $\mathcal{S}^{*}(\mathbb{R})$ with $\mathcal{S}(\mathbb{R}).$ In particular, if $f \in \mathcal{S}(\mathbb{R})$, and $\widehat{f}(u)=\int_{\mathbb{R}}e^{-iux}f(x)dx,~~\text{  }u\in\mathbb{R},$ then $\widehat{f}(u)\in \mathcal{S}(\mathbb{R}).$ Similarly, if $\widehat{f}(u)\in \mathcal{S}(\mathbb{R})$, and $f(x)=\int_{\mathbb{R}}e^{iux}\widehat{f}(u)du,~~\text{  }x\in\mathbb{R},$
then $f(x)\in \mathcal{S}(\mathbb{R})$, see \cite{stein}.

\noindent
Finally, we define Wasserstein-type distance, see \cite{k0}. Let
 $$\mathcal{H}_r = \left\{h:\mathbb{R}\to \mathbb{R}\bigg|h \mbox{ is $r$ times differentiable and},\|h^{(k)}\|\leq 1, k =0, 1,\ldots, r \right\},$$ where $h^{(k)}$, $k=1, \ldots,r$, is the $k$-th derivative of $h$, with $h^{(0)}=h$ and $\|f\|=\sup_{x\in\mathbb{R}}|f(x)|$. Then, for any two random variables $Y$ and $Z$ the distance is given by
\begin{equation*}
	d_{W_r}(Y,Z):=\sup_{h \in \mathcal{H}_r}\left|\mathbb{E}[h(Y)]-\mathbb{E}[h(Z)]\right|.
\end{equation*}
\noindent
We use this distance for studying TSD approximation problems. Note that, $d_{W_{r}}$ has the following order relationship with the classical Wasserstein distance $W_1$.  
$$d_{W_{r}}(Y,Z)\leq d_{W_{1}}(Y,Z)\leq W_{1}(Y,Z)\leq W_{p}(Y,Z),~~r,p\geq 1.$$ 
\noindent
 We use this relationship and discuss the consequences of our results in Section \ref{PP2:application}.

\section{Results}\label{mr}
\noindent
In this section, we present components of Stein's method for TSD.

\subsection{Stein characterization}First, we present a Stein characterization for TSD. 
\begin{thm}\label{th1}
	Let $X\sim \text{TSD}(\alpha^{+},\beta^{+},\lambda^{+};\alpha^{-},\beta^{-},\lambda^{-})$. Then,
	\begin{equation}\label{PP2:StenIdTSD}
	\mathbb{E}\left(Xf(X)-\displaystyle\int_{\mathbb{R}}f(X+u)\nu(du) \right)=0,~~f\in\mathcal{S}(\mathbb{R}).  
	\end{equation}	
\end{thm}

\begin{proof}

Recall first that, for $X\sim \text{TSD}(\alpha^{+},\beta^{+},\lambda^{+};\alpha^{-},\beta^{-},\lambda^{-})$, cf is given by \eqref{e1} with the L\'evy measure \eqref{e2}. Taking logarithms on both sides of (\ref{e1}), and differentiating with respect to $z$, we have




%


\begin{equation}\label{PP2:e4}
\phi^{\prime}(z)=i\int_{\mathbb{R}}ue^{izu}\nu(du)\phi(z).
\end{equation}

\noindent
Let $F_{X}$ be the distribution function (cumulative distribution function) of $X$. Then,

\begin{equation}\label{PP2:e5}
\phi(z)=\displaystyle\int_{\mathbb{R}}e^{izx}F_{X}(dx)~~\text{and}~~\phi^{\prime}(z)=i\displaystyle\int_{\mathbb{R}}xe^{izx}F_{X}(dx).
\end{equation}

\noindent
Using \eqref{PP2:e5} in \eqref{PP2:e4} and rearranging the integrals, we have

\begin{align}
\nonumber	0&=i\displaystyle\int_{\mathbb{R}}xe^{izx}F_{X}(dx)-i\int_{\mathbb{R}}ue^{izu}\nu(du)\phi(z)\\
	&=\displaystyle\int_{\mathbb{R}}xe^{izx}F_{X}(dx)-\int_{\mathbb{R}}ue^{izu}\nu(du)\phi(z)\label{PP2:e6}
\end{align}

\noindent
The second integral of \eqref{PP2:e6} can be written as

\begin{align}
	\nonumber \left(\int_{\mathbb R}ue^{izu}\nu(du)\right)\phi_X(z)&=\int_{\mathbb R}\int_{\mathbb R}ue^{izu}e^{izx}F_{X}(dx)\nu(du)\\
	\nonumber &=\int_{\mathbb R}\int_{\mathbb R}ue^{iz(u+x)}\nu(du)F_{X}(dx)\\
	\nonumber &=\int_{\mathbb R}\int_{\mathbb R}ue^{izy}\nu(du)F_{X}(d(y-u))\\
	\nonumber &=\int_{\mathbb R}\int_{\mathbb R}ue^{izx}\nu(du)F_{X}(d(x-u))\\ 
	&=\int_{\mathbb R}e^{izx}\int_{\mathbb R}uF_{X}(d(x-u))\nu(du).\label{PP2:e7}
\end{align}

\noindent
Substituting \eqref{PP2:e7} in \eqref{PP2:e6}, we have

\begin{align}
\nonumber	0&=\displaystyle\int_{\mathbb{R}}xe^{izx}F_{X}(dx)-\int_{\mathbb R}e^{izx}\int_{\mathbb R}uF_{X}(d(x-u))\nu(du)\\
	&=\displaystyle\int_{\mathbb{R}}e^{izx} \left( xF_{X}(dx)-\int_{\mathbb R}uF_{X}(d(x-u))\nu(du) \right) \label{PP2:e8}
\end{align}

\noindent
On applying Fourier transform to \eqref{PP2:e8}, multiplying with $f\in \mathcal{S}(\mathbb{R}),$ and integrating over $\mathbb{R},$ we get

\begin{align}\label{PP2:e9}
	\displaystyle\int_{\mathbb{R}}f(x) \left( xF_{X}(dx)-\int_{\mathbb R}uF_{X}(d(x-u))\nu(du) \right)=0.
\end{align}

\noindent
The second integral of \eqref{PP2:e9} can be seen as

\begin{align}
	\nonumber \int_{\mathbb R}\int_{\mathbb R}uf(x)F_{X}(d(x-u))\nu(du)&=\int_{\mathbb R}\int_{\mathbb R}uf(y+u)F_{X}(dy)\nu(du)\\ 
	\nonumber &=\int_{\mathbb R}\int_{\mathbb R}uf(x+u)F_{X}(dx)\nu(du)\\
	&=\mathbb{E}\left(\int_{\mathbb R}uf(X+u)\nu(du)\right).\label{PP2:e10}
\end{align}

\noindent
Substituting \eqref{PP2:e10} in \eqref{PP2:e9}, we have

\begin{align*}
	\mathbb{E}\left(Xf(X)-\displaystyle\int_{\mathbb{R}}f(X+u)\nu(du) \right)=0.
\end{align*}

\noindent
Hence the theorem is proved.
\end{proof}

\begin{rem}
	
	%
	One can also prove the converse of Theorem \ref{th1} by choosing $f(x)=e^{isx},$ where $s,x \in \mathbb{R}$ in (\ref{PP2:StenIdTSD}). We refer to the reader Appendix \ref{app} for proof of the converse of Theorem \ref{th1}. We derive the characterizing (Stein) identity (\ref{PP2:StenIdTSD}) for TSD using the L\'evy-Khinchine representation of the cf. Also, observe that several classes of distributions such as variance-gamma, bilateral-gamma, CGMY, and KoBol can be viewed as TSD. Stein characterization for these classes of distributions can be easily derived using (\ref{PP2:StenIdTSD}).
	
\end{rem}

\noindent
Note that, from Definition \ref{vgddf1} and Theorem \ref{th1}, for $f\in\mathcal{S}(\mathbb{R})$ a Stein identity for VGD$_0(\alpha,\lambda^{+},\lambda^{-})$ is 

\begin{align}
	\nonumber\mathbb{E}Xf(X)&=\mathbb{E}\left(\displaystyle\int_{\mathbb{R}}f(X+u)\nu_{\text{VGD}}(du) \right)\\
	&=\alpha\mathbb{E}\displaystyle\int_{0}^{\infty}\left(e^{-\lambda^{+}u}f(X+u) -e^{-\lambda^{-}u}f(X-u) \right)du.\label{PP2:e11}
\end{align}

\noindent
Next, we establish a Stein characterization for VGD$_{2}(\sigma^{2},r,\theta)$.
\begin{cor}\label{corvg}
	Let $X\sim VGD_{2}(\sigma^{2},r,\theta)$  with cf (\ref{cfvg2}). Then,
	\begin{equation}\label{PP2:SteinIdVGD}
	\mathbb{E}\left(\sigma^{2}Xf^{\prime\prime}(X)+\left(\sigma^{2}r+2\theta X \right)f^{\prime}(X)+\left(r\theta-X \right)f(X)  \right)=0,~~f\in\mathcal{S}(\mathbb{R}).
	\end{equation}
\end{cor}

\begin{proof} Applying integration by parts formula twice on the right hand side of \eqref{PP2:e11} and suitably adjusting the integrals, we have

\begin{align}
\nonumber	\mathbb{E}Xf(X)&=\alpha\left(\frac{1}{\lambda^{+}}-\frac{1}{\lambda^{-}}\right)\mathbb{E}f(X)\\
\nonumber	&+\alpha\left(\frac{1}{\lambda^{+}}-\frac{1}{\lambda^{-}}\right)\mathbb{E}\displaystyle\int_{0}^{\infty}\left(e^{-\lambda^{+}u}f^{\prime}(X+u) -e^{-\lambda^{-}u}f^{\prime}(X-u) \right)du\\
\nonumber	&+\frac{\alpha}{\lambda^{-}}\mathbb{E}\displaystyle\int_{0}^{\infty}e^{-\lambda^{+}u}f^{\prime}(X+u)du+\frac{\alpha}{\lambda^{+}}\mathbb{E}\displaystyle\int_{0}^{\infty}e^{-\lambda^{-}u}f^{\prime}(X-u)du\\
\nonumber	&=\alpha\left(\frac{1}{\lambda^{+}}-\frac{1}{\lambda^{-}}\right)\mathbb{E}f(X)\\
\nonumber	&+\alpha\left(\frac{1}{\lambda^{+}}-\frac{1}{\lambda^{-}}\right)\mathbb{E}\displaystyle\int_{0}^{\infty}\left(e^{-\lambda^{+}u}f^{\prime}(X+u) -e^{-\lambda^{-}u}f^{\prime}(X-u) \right)du\\
	&+\frac{2\alpha}{\lambda^{+}\lambda^{-}}\mathbb{E}f^{\prime}(X)+\frac{\alpha}{\lambda^{+}\lambda^{-}}\mathbb{E}\displaystyle\int_{0}^{\infty}\left(e^{-\lambda^{+}u}f^{\prime\prime}(X+u) -e^{-\lambda^{-}u}f^{\prime\prime}(X-u) \right)du.\label{PP2:e14}
\end{align}


\noindent
Next observe that
\begin{subequations}
	\begin{align}\label{PP2:e12}
		(a)~~\mathbb{E}Xf^{\prime}(X)&=\alpha\mathbb{E}\displaystyle\int_{0}^{\infty}\left(e^{-\lambda^{+}u}f^{\prime}(X+u) -e^{-\lambda^{-}u}f^{\prime}(X-u) \right)du,\\
		\label{PP2:e13} (b)~~\mathbb{E}Xf^{\prime\prime}(X)&=\alpha\mathbb{E}\displaystyle\int_{0}^{\infty}\left(e^{-\lambda^{+}u}f^{\prime\prime}(X+u) -e^{-\lambda^{-}u}f^{\prime\prime}(X-u) \right)du, 
	\end{align}
\end{subequations}
\noindent
as $f\in\mathcal{S}(\mathbb{R})$. Now, applying \eqref{PP2:e12} and \eqref{PP2:e13} on \eqref{PP2:e14}, we get

$$\mathbb{E}\left(\frac{1}{\lambda^{+}\lambda^{-}}Xf^{\prime\prime}(X)+\left(\frac{2\alpha}{\lambda^{+}\lambda^{-}}+\Lambda X \right)f^{\prime}(X)+\left(\alpha\Lambda-X \right)f(X)  \right)=0,$$

\noindent
where $\Lambda=\left(\frac{1}{\lambda^{+}}-\frac{1}{\lambda^{-}}\right),~~ f\in\mathcal{S}(\mathbb{R}).$ Setting the parameters
$$\frac{1}{\lambda^{+}\lambda^{-}}=\sigma^{2},~~  \Lambda=\left(\frac{1}{\lambda^{+}}-\frac{1}{\lambda^{-}}\right)=2\theta,~~\alpha=\frac{r}{2},$$
\noindent
we get the our desired conclusion. 
\end{proof}

\noindent
Next, we compare our characterization with some well-known Stein characterizations in literature.
\begin{rem}
	\begin{enumerate}
		\item [(i)] Our Stein characterization matches exactly with Stein characterization given in Gaunt \cite{k24}, whenever the location parameter $\mu=0$. In general, Gaunt \cite{k24} uses the density approach developed in \cite{den}, and the density of VGD is usually written in terms of modified Bessel function. Therefore, the derivation of Stein characterization using density approach is quite lengthy (see \cite{k24}). However, we show that using cf approach, the derivation of Stein characterization is quick and easy to understand.

		\item [(ii)] We also observe that a Stein identity for $VGD_{2}(\frac{\sigma^{2}}{r},r,0)$ is given by
		$$\mathbb{E}\left(\frac{\sigma^{2}}{r}Xf^{\prime\prime}(X)+\sigma^{2}f^{\prime}(X)-Xf(X) \right)=0,$$
		which in the limit $r\to\infty$ is the Stein identity for classical $\mathcal{N}(0,\sigma^{2}).$
		
		\item [(iii)] Taking $r=1,\sigma^{2}=\sigma_X^{2}\sigma_Y^{2}~~\text{and}~~\theta=0$, the Stein identity (\ref{PP2:SteinIdVGD}) reduces to
		$$\mathbb{E}\left(\sigma_X^{2}\sigma_Y^{2}\left(Xf^{\prime\prime}(X)+f^{\prime}(X)  \right)-Xf(X)  \right)=0,$$
		\noindent
		which is the Stein identity for products of independent $\mathcal{N}(0,\sigma_X^{2})$ and $\mathcal{N}(0,\sigma_Y^{2})$, see \cite{key1}.  
		
		\item [(iv)] We can also deduce Stein identities for symmetrized-gamma or symmetric case of variance-gamma, Laplace, gamma distributions using Corollary \ref{corvg}.
		%
		%
	\end{enumerate}
\end{rem}

\noindent
Next, we state a corollary for one-sided TSD, which provides a Stein characterization for gamma distribution.

\begin{cor}\label{cor1}
	Let $X\sim TSD_1(\alpha,0,\lambda)$. Then,
	
	\begin{equation}
	\mathbb{E}Xf(X)=\mathbb{E}X\mathbb{E}\left(f(X)+\frac{1}{\lambda}f^{\prime}(X+Y)\right),~~ f\in\mathcal{S}(\mathbb{R}),
	\end{equation}	
	\noindent
	where $Y$ is a random variable having exponential distribution with parameter $\lambda$, independent of X.	
\end{cor}

\begin{proof}
	\noindent
	Let $X\sim TSD_1(\alpha,0,\lambda)$ with $\alpha,\lambda>0$. Then by Remark \ref{p2}, $X$ has the gamma distribution with parameters $\alpha$ and $\lambda$. Following steps similar to the proof of Theorem \ref{th1}, one can find a Stein identity for $ TSD_1(\alpha,0,\lambda)$ in the form
	
	\begin{equation}\label{eq19}
	\mathbb{E}Xf(X)=\alpha\mathbb{E}\left(\displaystyle\int_{0}^{\infty}e^{-\lambda u}f(X+u)du\right),~~f\in \mathcal{S}(\mathbb{R}).
	\end{equation}

	\noindent
	Note that, $\mathbb{E}X=\frac{\alpha}{\lambda}$. Applying integration by parts formula on the right hand side of (\ref{eq19}), we have
	\begin{align*}
	\mathbb{E}Xf(X)&=\mathbb{E}\left(\frac{\alpha}{\lambda}f(X)+\frac{\alpha}{\lambda} \displaystyle\int_{0}^{\infty}e^{-\lambda u}f^{\prime}(X+u)du \right)\\
	&=\mathbb{E}X\mathbb{E}\left(f(X)+ \displaystyle\int_{0}^{\infty}e^{-\lambda u}f^{\prime}(X+u)du \right)\\
	&=\mathbb{E}X\mathbb{E}\left(f(X)+\frac{1}{\lambda}f^{\prime}(X+Y)\right),
	\end{align*}
	
	\noindent
	where $Y$ is exponential random variable with parameter $\lambda$, independent of $X$.

	\noindent
	Hence the result.
\end{proof}

\begin{rem}
	\begin{enumerate}
		\item [(i)]Note that Corollary \ref{cor1} claims the existence of an additive exponential size-bias (see, \cite{chen}) distribution for the gamma distribution.
		
		\item [(ii)]The Stein characterization for gamma distribution is first introduced by Luk (\cite[Subsection 2.2]{k5}) using Barbour generator approach \cite{k8} without additive size-bias distribution. The Stein identity given in ( \cite[Lemma 2.9]{k5}) is for $\chi^{2}_{2(n+1)}$ distribution with additive size-bias distribution. Under the assumptions of Luk \cite{k5}, both identities can be retrieved from Corollary \ref{cor1}.

	\end{enumerate}

\end{rem}

\subsection{Stein equation}
\noindent
Note that, from Theorem \ref{th1}, for any $f \in \mathcal{S}(\mathbb{R})$, $\mathcal{A}_{X}(f)(x):=-xf(x)+\displaystyle\int_{\mathbb{R}}f(x+u)\nu(du)$ is a Stein operator for TSD. Observe also that, $\mathcal{A}_{X}$ is an integral operator, where domain of the operator is $\mathcal{F}_X=\overline{\mathcal{S}(\mathbb{R})}$ (see, \cite{k20} for more details). For more general discussion on domain of operators, we refer the reader to \cite{stein} and references therein. As mentioned in Section \ref{Sec:Introduction}, the next step in Stein's method is to set a Stein equation. For any $X\sim\text{TSD}(\alpha^{+},\beta^{+},\lambda^{+};\alpha^{-},\beta^{-},\lambda^{-})$ and $h\in\mathcal{H}_{r}$ (see, Subsection \ref{PP2:FS}) with $\mathbb{E}h(X)<\infty$, a Stein equation for TSD is given by
\begin{equation}\label{PP2:e15}
	\mathcal{A}_{X}(f)(x)= h(x)-\mathbb{E}(h(X)).
\end{equation}




\noindent
To solve \eqref{PP2:e15}, we apply the semigroup approach. The semigroup approach for solving the Stein equation is developed by Barbour \cite{k8}, and Arras and Houdr\'e \cite{k0} generalized it for infinitely divisible distributions with the finite first moment. Following Barbour's approach \cite{k8}, we choose a family of operators $(P_{t})_{t\geq0}$, for all $x\in\mathbb{R}$, as
\begin{equation}\label{PP2:e16}
	P_{t}(f)(x)=\frac{1}{2\pi}\int_{\mathbb{R}}\hat{f}(\xi)e^{i\xi xe^{-t}}\frac{\phi(\xi)}{\phi(e^{-t}\xi)}d\xi, ~~f\in\mathcal{F}_{X}.
\end{equation} 
\noindent
 Note here that, one can define a cf, for all $z\in \mathbb{R},$ and $t\geq 0,$ by
\begin{align}\label{PP2:a15}
	\phi_t(z):=\frac{\phi(z)}{\phi(e^{-t}z)}=\displaystyle\int_{\mathbb{R}}e^{iz u}F_{X_{(t)}}(du),
\end{align}
\noindent
where $F_{X_{(t)}}$ is the distribution function of $X_{(t)}$ and $\phi$ is the cf of TSD given in \eqref{e1}. The property given in \eqref{PP2:a15} is also known as self-decomposability (see, \cite{sato}). Using this property, we get
\begin{align}
	\nonumber  P_t(f)(x)&=\frac{1}{2\pi}\int_{\mathbb R}\int_{\mathbb{R}}\widehat{f}(z)e^{iz xe^{-t}}e^{iz u}F_{X_{(t)}}(du)dz\\
	\nonumber  &=\frac{1}{2\pi}\int_{\mathbb R}\int_{\mathbb{R}}\widehat{f}(z)e^{iz (u+xe^{-t})}F_{X_{(t)}}(du)dz\\
	&=\displaystyle\int_{\mathbb{R}}f(u+xe^{-t})F_{X_{(t)}}(du),\label{PP2:a17}
\end{align}
\noindent
where the last step follows by applying inverse Fourier transform.


\begin{pro}\label{PP2:proSem}
	The family of operators $(P_{t})_{t\geq 0}$ given in \eqref{PP2:e16} is a $\mathbb{C}_0$-semigroup on $\mathcal{F}_{X}$.
\end{pro}
\noindent
For details of the proof, we refer the reader to Appendix \ref{app}.\vspace{2mm}

\noindent
Next, we establish an infinitesimal generator of the semigroup $(P_t)_{t\geq 0}$.

\begin{lem}\label{PP2:th2}
 Let $(P_{t})_{t\geq0}$ be a $\mathbb{C}_{0}$ semigroup defined in \eqref{PP2:e16}. Then, its generator $\mathcal{T}$ is given by
	\begin{equation}\label{e7}
	\mathcal{T}(f)(x)=-xf^{\prime}(x)+\displaystyle\int_{\mathbb R}f^{\prime}(x+u)\nu(du),~~f\in\mathcal{S}(\mathbb{R}).
	\end{equation}	
\end{lem}

\begin{proof}
	For all $f\in \mathcal{S}(\mathbb{R}),$
	\begin{align}
		\nonumber \mathcal{T}(f)(x)&=\lim_{t\to0^{+}}\frac{1}{t}\left(P_{t}(f)(x) -f(x)   \right)\\
		\nonumber &=\frac{1}{2\pi}\lim_{t\to0^{+}}\int_{\mathbb R}\widehat{g}(z)e^{iz x} \frac{1}{t}\big(e^{iz x(e^{-t}-1)}\phi_{t}(z)-1\big)dz\\ 
		\nonumber &= \frac{1}{2\pi}\int_{\mathbb R}\widehat{f}(z)e^{iz x}\left(-x+ \displaystyle\int_{\mathbb{R}}e^{iz u }u\nu(du) \right )(iz)dz \text{ (using Prop. \ref{PP2:appendixPro1})} \\
		\nonumber&=-xf^{\prime}(x)+\int_{\mathbb{R}}f^{\prime}(x+u)u\nu(du),
	\end{align}
\noindent
where the last equality follows by applying inverse Fourier transform.

\noindent
This completes the proof.
\end{proof}

\noindent
Observe that, for any $f\in\mathcal{S}(\mathbb{R})$, 
\begin{align*}
	\mathcal{T}f(x)&=-xf^{\prime}(x)+\int_{\mathbb{R}}f^{\prime}(x+u)u\nu(du)\\
	&=\mathcal{A}_{X}(f^{\prime})(x).
\end{align*}

\noindent
Next, we provide the solution to our Stein equation \eqref{PP2:e15}.  

\begin{thm}\label{thmsol}
	Let $X\sim \text{TSD}(\alpha^{+},\beta^{+},\lambda^{+};\alpha^{-},\beta^{-},\lambda^{-})$. Then for $h\in\mathcal{H}_{r}$, the function $f_{h}:\mathbb{R}\to \mathbb{R}$ defined by 
	\begin{equation}\label{PP2:SolSe}
		f_{h}(x):=-\displaystyle
		\int_{0}^{\infty}e^{-t}\int_{\mathbb{R}}h^{\prime}(xe^{-t}+y)F_{X_{(t)}}(dy)dt,
	\end{equation}
	solves \eqref{PP2:e15}.

\end{thm}

\begin{proof}

 To prove this theorem, we use the connection between the operators $\mathcal{A}_X$ and $\mathcal{T}$. We write,
\begin{align}
	\nonumber\mathcal{A}f_{h}(x)&=-xf_{h}(x)+\int_{\mathbb{R}} f_{h}(x+u)u\nu (du)\\
	\nonumber&=\mathcal{T}(g_{h})(x),~~(\text{where }g_{h}(x)=-\displaystyle\int_{0}^{\infty}\left(P_{t}(h)(x)-\mathbb{E}h(X)  \right)dt,~h\in \mathcal{H}_r ) \\
	\nonumber&=-\displaystyle\int_{0}^{\infty}\mathcal{T}P_{t}(h)(x)dt\\
	\nonumber&=-\displaystyle\int_{0}^{\infty}\frac{d}{ds}P_{t}(h)(x)dt\\
	\nonumber&=P_{0}h(x)-P_{\infty}h(x)\\
	\nonumber&=h(x)-\mathbb{E}h(X)~(\text{by Proposition \ref{PP2:proSem}}).
\end{align}
\noindent
Hence, $f_{h}$ is the solution to \eqref{PP2:e15}. 

\noindent
 Using some standard argument, one can show that $g_{h}$ is well-defined and $g_{h}^{\prime}(x)=f_{h}(x)$, $x\in\mathbb{{R}}$. For details of the proof, we refer the reader to Appendix \ref{app}.
\end{proof}

\subsection{Properties to the solution}
\noindent
The next step is to estimate the properties of $f_{h}$. In the following theorem, we establish estimates of $f_{h}$, which play a crucial role in the TSD approximation problems. 	Gaunt \cite{k24,kk2} and D$\ddot{o}$bler et. al. \cite{kk6} propose various methods for bounding the solution to the Stein equations that allow them to derive properties of the solution to the Stein equation, in particular for a subfamily of TSD, namely the variance-gamma. However, we derive the properties of the solution to the Stein equation for TSD using its self-decomposable property.
\begin{thm}\label{th3}
	For $h\in\mathcal{H}_4$, let $f_{h}$ be defined in \eqref{PP2:SolSe}. Then,
		\begin{align}\label{PP2:pr1}
			\|f_{h}\|\leq\|h^{(1)}\|,~~\|f^{\prime}_{h}\|\leq\frac{1}{2}\|h^{(2)}\|,~~
			\|f^{\prime\prime}_{h}\|\leq\frac{1}{3}\|h^{(3)}\|, \|f^{\prime\prime\prime}_{h}\|\leq\frac{1}{4}\|h^{(4)}\|.
		\end{align}	
	\noindent
	For any $x,y\in\mathbb{R},$ 
	\begin{align}\label{PP2:pr02}
		\|f^{\prime\prime}_{h}(x)- f^{\prime\prime}_{h}(y)  \| \leq \frac{\|h^{(4)}\|}{4}\left| x-y\right|.
	\end{align}	 
	\end{thm}


\begin{proof}

Recall the definition of $(P_{t})_{t\geq 0}$, 

$$P_{t}f(x)=\displaystyle\int_{\mathbb{R}}f(y+e^{-t}x)F_{X_{(t)}}(dy),~~f\in\mathcal{F}_{X},$$

\noindent
where $F_{X_{(t)}}$ is the distribution function of $X_{(t)}$. Thus, for $h\in\mathcal{H}_{4}$,
\begin{align*}
	\frac{d}{dx}(P_{t}(h)(x))&=e^{-t}\int_{\mathbb R}h^{(1)}(xe^{-t}+y)F_{X_{(t)}}(dy),\\ 
	\frac{d^{2}}{dx^{2}}(P_{t}(h)(x))&=e^{-2t}\int_{\mathbb R}h^{(2)}(xe^{-t}+y)F_{X_{(t)}}(dy),\\
	 \frac{d^{3}}{dx^{3}}(P_{t}(h)(x))&=e^{-3t}\int_{\mathbb R}h^{(3)}(xe^{-t}+y)F_{X_{(t)}}(dy),\\
	 \text{and}~~\frac{d^{4}}{dx^{4}}(P_{t}(h)(x))&=e^{-4t}\int_{\mathbb R}h^{(4)}(xe^{-t}+y)F_{X_{(t)}}(dy).
\end{align*}

\noindent
Let 

$$	f_{h}(x)=-\displaystyle
\int_{0}^{\infty}e^{-t}\int_{\mathbb{R}}h^{\prime}(xe^{-t}+y)F_{X(t)}(dy)dt.$$

\noindent
It can be easily seen that $f_{h}$ is thrice differentiable.
Hence, $\|f_{h}\|\leq\|h^{(1)}\|,~~\|f^{\prime}_{h}\|\leq\frac{1}{2}\|h^{(2)}\|,~~
\|f^{\prime\prime}_{h}\|\leq\frac{1}{3}\|h^{(3)}\|, \|f^{\prime\prime\prime}_{h}\|\leq\frac{1}{4}\|h^{(4)}\|$.\vspace{2mm}

\noindent
Now observe that, for any $x,y\in\mathbb{R}$ and $h\in \mathcal{H}_{4}$,

\begin{align*}
	\left|f^{\prime\prime}_{h}(x)-f^{\prime\prime}_{h}(y) \right| & \leq \displaystyle\int_{0}^{\infty}e^{-3t}\int_{\mathbb{R}}\left|h^{(3)}(xe^{-t}+z)-h^{(3)}(ye^{-t}+z)  \right|F_{X_{(t)}}(dz)dt\\
	&\leq  \displaystyle\int_{0}^{\infty}e^{-3t}\int_{\mathbb{R}}\|h^{(4)}\|\left|x-y\right|e^{-t}F_{X_{(t)}}(dz)dt\\
	&=\|h^{(4)}\|\left|x-y  \right| \displaystyle\int_{0}^{\infty}e^{-4t}dt\\
	&=\frac{\|h^{(4)}\|}{4}\left|x-y  \right|,
\end{align*}
\noindent
the desired conclusion follows.

\end{proof}


\section{Applications}\label{PP2:application}
\noindent
In this section, we present three applications of our estimates.

\subsection{Comparison between two TSD} As a first application of our estimates, we derive a simple error bound for approximation between TSD. We refer the reader to \cite{ley} for a number of similar bounds for comparison of uni-variate distributions. First, we establish a corollary to Theorem \ref{th3}, which is used in deriving an upper bound in the Wasserstein-type distance between two TSD.

\begin{cor}
	\noindent
	For $h\in\mathcal{H}_3$, let $f_{h}$ be defined in \eqref{PP2:SolSe}. Let $\alpha^{+}=\alpha^{-}=\alpha$, $\beta^{+}=\beta^{-}=0$ and $\lambda^{+}=\lambda^{-}=\lambda$. Then, for any $x\in\mathbb{R}$
	\begin{align}\label{PP2:pr2}
		\|xf_{h}^{\prime\prime} (x)\| \leq 2\left(\|h^{(2)}\|+\frac{\alpha}{3\lambda}\|h^{(3)}\|\right).
	\end{align}
\end{cor} 

\begin{proof}
\noindent
As $f_{h}$ solves \eqref{PP2:e15}, thus we have

\begin{align}\label{PP2:e35}
	-xf_{h}(x)+\int_{\mathbb{R}}f_{h}(x+u)\nu(du)=h(x)-\mathbb{E}h(X).
\end{align}

\noindent
Recall that, the L\'evy measure for TSD($\alpha^{+},\beta^{+},\lambda^{+};\alpha^{-},\beta^{-},\lambda^{-}$) given in \eqref{e2}. Assume that $\alpha^{+}=\alpha^{-}=\alpha$, $\beta^{+}=\beta^{-}=0$, $\lambda^{+}=\lambda^{-}=\lambda$ and differentiating \eqref{PP2:e35} twice with respect to $x$, we have

\begin{align}
	\nonumber	-xf^{\prime\prime}_{h}(x)&=h^{(2)}(x)+2f^{\prime}_{h}(x)-\displaystyle\int_{\mathbb{R}}f^{\prime\prime}_{h}(x+u)\nu(du)\\
	&=h^{(2)}(x)+2f^{\prime}_{h}(x)-\alpha\displaystyle\int_{0}^{\infty}e^{-\lambda u} \left(f_{h}^{\prime\prime}(x+u)- f_{h}^{\prime\prime}(x-u)  \right)du.\label{PP2:e36}
\end{align}

\noindent
Using \eqref{PP2:pr1}, we have

\begin{align*}
	\|xf_{h}^{\prime\prime}(x)\|&\leq 2\|h^{(2)}\|+\frac{2\alpha}{3}\|h^{(3)}\|\int_{0}^{\infty}e^{-\lambda u}du\\
	&=2\|h^{(2)}\|+\frac{2\alpha}{3}\frac{\|h^{(3)}\|}{\lambda},
\end{align*}

\noindent
the desired conclusion follows.

\end{proof}

\noindent
In the following theorem, we establish a simple error bound for approximation between two TSD.

\begin{thm}
	Let $X\sim\text{TSD}(\alpha_{1},0,\lambda_{1};\alpha_{1},0,\lambda_{1})$ and $Y\sim\text{TSD}(\alpha_{2},0,\lambda_{2};\alpha_{2},0,\lambda_{2})$. Then, for $\lambda_{1}>1$
	\begin{align*}
		d_{W_{3}}(Y,X) \leq \frac{\lambda_{1}^{2}}{\lambda_{1}^{2}-1}\left( \left|\frac{\alpha_{1}}{\lambda_{1}^{2}}- \frac{\alpha_{2}}{\lambda_{2}^{2}} \right|+2\left(1+\frac{\alpha_1}{3\lambda_1}\right)\left| \frac{1}{\lambda_{1}^{2}}-\frac{1}{\lambda_{2}^{2}} \right| \right).
	\end{align*}
\end{thm}
 
 \begin{proof}
 	By \eqref{PP2:e15}, we can write the Stein equation for $X\sim\text{TSD}(\alpha_{1},0,\lambda_{1};\alpha_{1},0,\lambda_{1})$ as 
 	\begin{align}
 		\nonumber	h(x)-h(X)&=-xf(x)+\displaystyle\int_{\mathbb R}f(x+u)\nu(du)\\
 		&=-xf(x)+\alpha_1 \displaystyle\int_{0}^{\infty}\left(f(x+u)-f(x-u) \right)e^{-\lambda_1 u}du \label{PP2:a24}
 	\end{align}

 	\noindent
 	Thus, we have
 	\begin{align}
 		\nonumber	\mathbb{E}\left(h(Y)-h(X) \right)&=\mathbb{E}\left(-Yf_{h}(Y)+\alpha_1 \displaystyle\int_{0}^{\infty}\left(f_{h}(Y+u)-f_{h}(Y-u) \right)e^{-\lambda_1 u}du \right)\\
 		\nonumber	&=\mathbb{E}\left(-Yf_{h}(Y)+\frac{\alpha_{1}}{\lambda_{1}}\displaystyle\int_{0}^{\infty}\left(f_{h}^{\prime}(Y+u)+f_{h}^{\prime}(Y-u) \right) e^{-\lambda_1 u}du \right)\\
 		\nonumber	&=\mathbb{E}\left(-Yf_{h}(Y)+\frac{2\alpha_1}{\lambda_{1}^{2}}f_{h}^{\prime}(Y)+ \frac{\alpha_{1}}{\lambda_{1}^{2}}\int_{0}^{\infty}\left(f_{h}^{\prime\prime}(Y+u)-f_{h}^{\prime\prime}(Y-u) \right) e^{-\lambda_1 u}du\right)\\
 		\nonumber	&=\mathbb{E}\left(-Yf_{h}(Y)+\frac{2\alpha_1}{\lambda_{1}^{2}}f_{h}^{\prime}(Y)+\frac{1}{\lambda_{1}^{2}}Yf_{h}^{\prime\prime}(Y)\right)\\&  + \frac{1}{\lambda_{1}^{2}}\mathbb{E}\left(-Yf^{\prime\prime}_{h}(Y)+\alpha_{1}\displaystyle\int_{0}^{\infty}\left(f_{h}^{\prime\prime}(Y+u)-f_{h}^{\prime\prime}(Y-u) \right) e^{-\lambda_1 u}du\right) \label{PP2:a25}
 	\end{align}
 	
 	\noindent
 	Taking $\sup_{h\in\mathcal{H}_{3}}$ on both side of \eqref{PP2:a25} and rearranging the terms, we have
 	
 	\begin{align}\label{PP2:a26}
 		\left(1-\frac{1}{\lambda_{1}^{2}} \right)d_{W_{3}}(Y,X) \leq  \sup_{h\in\mathcal{H}_{3}}\left|\mathbb{E}\left(-Yf_{h}(Y)+\frac{2\alpha_1}{\lambda_{1}^{2}}f_{h}^{\prime}(Y)+\frac{1}{\lambda_{1}^{2}}f_{h}^{\prime\prime}(Y)   \right)  \right|
 	\end{align}
 	
 	\noindent
 	Note that, $Y\sim\text{TSD}(\alpha_{2},0,\lambda_{2};\alpha_{2},0,\lambda_{2}).$ Hence, following steps similar to proof of Corollary \ref{corvg}, we get
 	
 	\begin{align}\label{PP2:a27}
 		\mathbb{E} \left(-Yf_{h}(Y)+\frac{2\alpha_2}{\lambda_{2}^{2}}f_{h}^{\prime}(Y)+\frac{1}{\lambda_{2}^{2}}f_{h}^{\prime\prime}(Y)   \right)=0
 	\end{align}
 	\noindent
 	Using \eqref{PP2:a27} in \eqref{PP2:a26}, we have

 	\begin{align}
 		\nonumber	\left(1-\frac{1}{\lambda_{1}^{2}} \right)d_{W_{3}}(Y,X)  & \leq \sup_{h\in\mathcal{H}_{3}}\left|\mathbb{E}\left(-Yf_{h}(Y)+\frac{2\alpha_1}{\lambda_{1}^{2}}f_{h}^{\prime}(Y)+\frac{1}{\lambda_{1}^{2}}f_{h}^{\prime\prime}(Y)   \right)\right.\\
 \nonumber	&	\left.-\mathbb{E}\left(-Yf_{h}(Y)+\frac{2\alpha_2}{\lambda_{2}^{2}}f_{h}^{\prime}(Y)+\frac{1}{\lambda_{2}^{2}}f_{h}^{\prime\prime}(Y)   \right)  \right| \\
 		\nonumber& \leq \sup_{h\in\mathcal{H}_{3}} \left|\mathbb{E}\left[ \left(\frac{2\alpha_1}{\lambda_{1}^{2}}-\frac{2\alpha_2}{\lambda_{2}^{2}}  \right)f^{\prime}_{h}(Y)\right]+\mathbb{E}\left[\left(\frac{1}{\lambda^{2}_{1}}-\frac{1}{\lambda^{2}_{2}} \right)Yf^{\prime\prime}_{h}(Y)\right]  \right|\\
 		\label{PP2:a28}	&\leq \left|\frac{2\alpha_1}{\lambda_{1}^{2}}-\frac{2\alpha_2}{\lambda_{2}^{2}}\right| \|f^{\prime}_{h}\|+\left|\frac{1}{\lambda^{2}_{1}}-\frac{1}{\lambda^{2}_{2}} \right| \|yf_{h}^{\prime\prime}(y)\|
 	\end{align}
 	
 	\noindent
 	Assume that $\|h^{(k)}\|\leq 1$ for $k=2,3$. Then, by using \eqref{PP2:pr1} and \eqref{PP2:pr2} in \eqref{PP2:a28}, we get our desired result.
 	
 \end{proof}

\begin{rem}
	Note that, if $\alpha_{1}=\alpha_{2}$ and $\lambda_{1}=\lambda_{2}>1$, this bound is equal to zero. This shows that $Y\overset{d}{=}X.$
\end{rem}

\subsection{Rate of convergence for the Laplace approximation of random geometric sums } Recall that the distribution of a random variable $X$ with cf $\psi_1(z)=\frac{1}{1+\frac{z^{2}}{\lambda^{2}}}, ~~z\in\mathbb{R},$ where $\lambda>0$, is called the Laplace distribution (we write $X\sim Laplace (0,\frac{1}{\lambda^{2}})$). Note here that, $X\sim Laplace(0,\frac{1}{\lambda^{2}}) \overset{d}{=}\text{TSD}(1,0,\lambda;1,0,\lambda)$. Let $N_p$ be a $Geo(p)$ random variable with cf $\psi_2(z)=\frac{p}{1-(1-p)e^{iz}}, ~~z\in\mathbb{R}.$ In the following theorem, we present a well-known limit theorem concerning to geometric sum of i.i.d random variables.

\begin{thm}(\cite{gaunt7}, p.2)
	Let $(Y_{n})_{n\geq 1}$ be a sequence of i.i.d random variables with zero mean and variance $\frac{2}{\lambda^{2}}\in (0,\infty)$ and let $N_p \sim Geo(p)$ be independent of $Y_{i}$ with probability mass function $\mathbb{P}(N_{p}=k)=p(1-p)^{k};$ $k=0,1,2,\ldots,$ $0<p<1$. Then, $S_{p}:=\sqrt{p}\sum_{i=1}^{N_{p}}Y_{i} \overset{d}{\to}Laplace(0,\frac{1}{\lambda^{2}}),~~\text{as}~~p\to0$.
\end{thm}

\noindent
Next, we define centered equilibrium distribution, which plays an important role in the Laplace approximation for geometric random sum. 

\begin{defn}(\cite[p.9]{pike})
	For any non-degenerate random variable $Y$ with mean zero and variance $\frac{2}{\lambda^{2}}\in (0,\infty)$, we say that the random variable $Y^{\text{L}}$ has the centered equilibrium distribution with respect to $Y$ if 
	
	\begin{align}\label{PP2:a29}
		\mathbb{E}g(Y)- g(0)=\frac{1}{\lambda^{2}}\mathbb{E}\left( g^{\prime\prime}(Y^{\text{L}}) \right),
	\end{align} 
for all twice differentiable functions $g$ such that $g,g^{\prime}$ and $g^{\prime\prime}$ are bounded. We can the map $Y\to Y^{\text{L}}$ the centered equilibrium transformation.
\end{defn}
\noindent
Note here that, if we consider $g(x)=xf(x) \in \mathcal{S}(\mathbb{R})$, then \eqref{PP2:a29} becomes

\begin{align}\label{PP2:a30}
	\mathbb{E}Yf(Y)=\frac{1}{\lambda^{2}}\mathbb{E} \left(Y^{\text{L}}f^{\prime\prime}(Y^{\text{L}})+2f^{\prime}(Y^{\text{L}})   \right).
\end{align}

\noindent
To derive upper bound for the Laplace approximation in the Wasserstein-type distance, we need the following lemma.

\begin{lem}\label{PP2:lemCED}
	Let $(Y_{n})_{n\geq 1}$ be a sequence of i.i.d random variables such that $\mathbb{E}Y_{i}=0$ and $\mathbb{E}Y_{i}^{2}=\frac{2}{\lambda^{2}}\in (0,\infty)$ and $\sup_{i\in \mathbb{N}}\mathbb{E}|Y_{i}^{3}|=\rho <\infty$, and let $N_{p}\sim Geo (p)$ be a random variable independent of $Y_{i}.$ Then
	\begin{itemize}
		\item [1.] For $S_{p}:=\sqrt{p}\sum_{k=1}^{N_{p}}Y_{k}$, the variable with centered equilibrium distribution has the form $S_{p}^{\text{L}}=\sqrt{p}\left( \sum_{k=1}^{N_{p}}Y_k+Y_{N_{p}+1}^{\text{L}} \right)$.
		
		\item [2.] $\mathbb{E} \left|S_{p}-S_{p}^{\text{L}}  \right|=\frac{\lambda^{2}\rho}{6}\sqrt{p}.$
	\end{itemize}
\end{lem}

\noindent
The proof of this lemma follows by similar computations \cite[Lemma 4.1]{slepov} and \cite[Proposition 3.4]{pike}, respectively.\vspace{2mm}

\noindent
Next, we establish a corollary to Theorem \ref{th3}, which is used in deriving rate of convergence for the Laplace approximation of geometric sums.

\begin{cor}
	\noindent
For $h\in\mathcal{H}_4$, let $f_{h}$ be defined in \eqref{PP2:SolSe}. Let $\alpha^{+}=\alpha^{-}=1$, $\beta^{+}=\beta^{-}=0$ and $\lambda^{+}=\lambda^{-}=\lambda$. Define $A_0f_{h}(x)=xf_{h}(x).$ Then, for any $x,y\in\mathbb{R}$
	\begin{align}\label{PP2:pr3}
		\| (A_0f_{h}(x))^{\prime\prime}-(A_0f_{h}(y))^{\prime\prime} \| \leq \left(\|h^{(3)}\|+\frac{\|h^{(4)}\|}{2\lambda} \right)|x-y|.
	\end{align}
\end{cor}

\begin{proof}

	\noindent
	Recall that, the L\'evy measure for TSD($\alpha^{+},\beta^{+},\lambda^{+};\alpha^{-},\beta^{-},\lambda^{-}$) given in \eqref{e2}. Let $\alpha^{+}=\alpha^{-}=1$, $\beta^{+}=\beta^{-}=0$ and $\lambda^{+}=\lambda^{-}=\lambda$. Define $A_0f_{h}(x)=xf_{h}(x).$ Then,  differentiating \eqref{PP2:e35} twice with respect to $x$, we have 
	
	$$(A_0f_{h}(x))^{\prime\prime}=-h^{(2)}(x)+\displaystyle\int_{0}^{\infty}e^{-\lambda u} \left(f_{h}^{\prime\prime}(x+u)- f_{h}^{\prime\prime}(x-u)  \right)du.$$
	
	\noindent
	For any $x,y\in\mathbb{R},$
	
	\begin{align}
		\nonumber	\left|(A_0f_{h}(x))^{\prime\prime}-(A_0f_{h}(y))^{\prime\prime}  \right| &\leq \left|h^{(2)}(x)-h^{(2)}(y)  \right|+\left|\displaystyle\int_{0}^{\infty}e^{-\lambda u} (f_{h}^{\prime\prime}(x+u)-f_{h}^{\prime\prime}(y+u))du  \right|\\
		\nonumber	&+\left|\displaystyle\int_{0}^{\infty}e^{-\lambda u}(f_{h}^{\prime\prime}(x-u)-f_{h}^{\prime\prime}(y-u)) du \right|\\
		\nonumber	& \leq \|h^{(3)}\||x-y|+\frac{\|h^{(4)}\|}{4\lambda}|x-y|+\frac{\|h^{(4)}\|}{4\lambda}|x-y|\\
		&=\left(\|h^{(3)}\|+\frac{\|h^{(4)}\|}{2\lambda} \right)|x-y|,
	\end{align}

	\noindent
	where the last but one inequality follows by using \eqref{PP2:pr02}.

\end{proof}

\noindent
In the following theorem, we obtain an error bound for the Laplace approximation of random geometric sums.

\begin{thm}
	Let $(Y)_{n\geq 1}$ be a sequence of i.i.d random variables with $\mathbb{E}Y_i=0$, $\mathbb{E}Y_{i}^{2}=\frac{2}{\lambda^{2}}\in (0,\infty)$ and let $N_p \sim Geo(p)$ be independent of $Y_{i}$ with probability mass function $\mathbb{P}(N_{p}=k)=p(1-p)^{k};$ $k=0,1,2,\ldots,$ $0<p<1$. Also let $X\sim Laplace (0,\frac{1}{\lambda^{2}})$ and denote $S_{p}:=\sqrt{p}\sum_{i=1}^{N_{p}}Y_{i}$. Then, for any $\lambda>1$,
	\begin{align}\label{PP2:rc}
		d_{W_{4}}(S_{p},X) \leq  \frac{\rho\lambda(2\lambda +1)}{12(\lambda^{2}-1)} p^{\frac{1}{2}}.
	\end{align}
\end{thm}

\begin{proof}
	Recall first that $Laplace(0,\frac{1}{\lambda^{2}})\overset{d}{=}\text{TSD}(1,0,\lambda;1,0,\lambda)$.
	
	\noindent
	Thus, using \eqref{PP2:e15}, we have
	\begin{align}
	\nonumber	\mathbb{E}\left(h(S_{p})-h(X)  \right)&=\mathbb{E}\left(-S_{p}f_{h}(S_{p})+\displaystyle\int_{\mathbb{R}}f_{h}(S_{p}+u)\nu(du)  \right)\\
	\nonumber	&=\mathbb{E}\left(-S_{p}f_{h}(S_{p})+ \displaystyle\int_{0}^{\infty}\left(f_{h}(S_{p}+u)-f_{h}(S_{p}-u) \right)e^{-\lambda u}du \right)\\
		\nonumber	&=\mathbb{E}\left(-S_{p}f_{h}(S_{p})+\frac{1}{\lambda}\displaystyle\int_{0}^{\infty}\left(f_{h}^{\prime}(S_{p}+u)+f_{h}^{\prime}(S_{p}-u) \right) e^{-\lambda u}du \right)\\
		\nonumber	&=\mathbb{E}\left(-S_{p}f_{h}(S_{p})+\frac{2}{\lambda^{2}}f_{h}^{\prime}(S_{p})+ \frac{1}{\lambda^{2}}\displaystyle\int_{0}^{\infty}\left(f_{h}^{\prime\prime}(S_{p}+u)-f_{h}^{\prime\prime}(S_{p}-u) \right) e^{-\lambda u}du \right)\\
		\nonumber	&=\mathbb{E}\left(-S_{p}f_{h}(S_{p})+\frac{2}{\lambda^{2}}f_{h}^{\prime}(S_{p})+\frac{1}{\lambda^{2}}S_{p}f_{h}^{\prime\prime}(S_{p})\right)\\
		\nonumber&  + \frac{1}{\lambda^{2}}\mathbb{E}\left(-S_{p}f^{\prime\prime}_{h}(S_{p})+\displaystyle\int_{0}^{\infty}\left(f_{h}^{\prime\prime}(S_{p}+u)-f_{h}^{\prime\prime}(S_{p}-u) \right) e^{-\lambda u}du\right)\\
	\nonumber	&=\mathbb{E}\left(-S_{p}f_{h}(S_{p})+\frac{2}{\lambda^{2}}f_{h}^{\prime}(S_{p})+\frac{1}{\lambda^{2}}S_{p}f_{h}^{\prime\prime}(S_{p})\right)\\&  + \frac{1}{\lambda^{2}}\mathbb{E}\left(-S_{p}f^{\prime\prime}_{h}(S_{p})+\displaystyle\int_{\mathbb{R}}f_{h}^{\prime\prime}(S_{p}+u)\nu(du)\right)\label{PP2:a31}
	\end{align} 

\noindent
Taking $\sup_{h\in\mathcal{H}_{4}}$ on both side of \eqref{PP2:a31} and rearranging the terms, we have

\begin{align}\label{PP2:a32}
	\left(1-\frac{1}{\lambda^{2}} \right)d_{W_{4}}(S_{p},X) \leq \sup_{h\in\mathcal{H}_{4}}\left|\mathbb{E}\left(-S_{p}f_{h}(S_{p})+\frac{2}{\lambda^{2}}f_{h}^{\prime}(S_{p})+\frac{1}{\lambda^{2}}f_{h}^{\prime\prime}(S_{p})   \right)  \right|
\end{align}

\noindent
Using \eqref{PP2:a30} in \eqref{PP2:a32}, we have

\begin{align}
\nonumber	d_{W_{4}}(S_{p},X)  &\leq \sup_{h\in\mathcal{H}_{4}}\frac{\lambda^{2}}{\lambda^{2}-1}\left(\mathbb{E}\left[\frac{2}{\lambda^{2}}f_{h}^{\prime}(S_{p})+\frac{1}{\lambda^{2}}f_{h}^{\prime\prime}(S_{p})\right]-\mathbb{E}\left[\frac{2}{\lambda^{2}}f_{h}^{\prime}(S_{p}^{\text{L}})+\frac{1}{\lambda^{2}}f_{h}^{\prime\prime}(S_{p}^{\text{L}})\right]   \right)\\
	&=\sup_{h\in\mathcal{H}_{4}}\frac{1}{\lambda^{2}-1}\left(\mathbb{E}\left[2f_{h}^{\prime}(S_{p})+f_{h}^{\prime\prime}(S_{p})\right]-\mathbb{E}\left[2f_{h}^{\prime}(S_{p}^{\text{L}})+f_{h}^{\prime\prime}(S_{p}^{\text{L}})\right]   \right)\label{PP2:a33}
\end{align}

\noindent
Assume that $\|h^{(k)}\|\leq 1$ for $k=3,4$. Then, by using \eqref{PP2:pr3} on \eqref{PP2:a33}, we have
\begin{align*}
	\nonumber	d_{W_{4}}(S_{p},X)  &\leq \frac{1}{\lambda^{2}-1}\left( 1+\frac{1}{2\lambda}\right)\mathbb{E}\left| S_{p}-S_{p}^{\text{L}} \right|\\
	&\leq \frac{2\lambda +1}{2\lambda(\lambda^{2}-1)}\times \frac{\lambda^{2}}{6}\rho p^{\frac{1}{2}} ~~(\text{by Lemma \ref{PP2:lemCED}})\\
	&=\frac{\rho\lambda(2\lambda +1)}{12(\lambda^{2}-1)} p^{\frac{1}{2}},
\end{align*}

\noindent
the desired conclusion follows.
\end{proof}
\begin{rem}
	 The reference \cite{pike} shows that $d_{BL}(S_{p},X)\leq C_{1}(\lambda,\rho)p^{\frac{1}{2}}$, where $C_{1}(\lambda,\rho)$ is some positive constant depends on $\lambda$ and $\rho$, and $d_{BL}$ denotes the bounded Lipschitz distance. Gaunt \cite{kk2} also proves that $W_{1}(S_{p},X) \leq C_{2}(\lambda,\rho)p^{\frac{1}{2}},$ where $C_{2}(\lambda,\rho)$ is a positive constant depends on $\lambda$ and $\rho$. In comparison with the rates derived in \cite{pike,kk2}, we see that the $O(p^{\frac{1}{2}})$ rate in \eqref{PP2:rc} is optimal.

\end{rem}

\subsection{ Six moment theorem for the symmetric variance gamma approximation of double Wiener-It$\hat{\text{o}}$ integrals}

Recently, Eichelsbacher and Th$\ddot{\text{a}}$le \cite{key2} extended the Malliavin-Stein method for variance gamma distributions. Here, we obtain an upper bound for the symmetric variance gamma approximation of general functional of an isonormal Gaussian process in the Wasserstein-type distance. We also prove the six moment theorem for the symmetric variance gamma approximation of double Wiener-It$\hat{\text{o}}$ integrals.\vspace{2mm}

\noindent
Let us first introduce some notation (see, the book \cite{nourdin} for detailed discussion). Let $\mathbb{D}^{p,q}$ be the Banach space of all functions in $L^{q}(\gamma)$, where $\gamma$ is the standard Gaussian measure, whose Malliavin derivative up to order $p$ also belong to $L^{q}(\gamma)$. Let $\mathbb{D}^{\infty}$ be the class of infinitely many times Malliavin differentiable random variables. We also introduce the well-known $\Gamma$-operators (see, \cite{nourdin0}). For a random variable $G\in\mathbb{D}^{\infty}$, we define $\Gamma_1(G)=G$, and for every $j\geq2,$

$$\Gamma_{j}(G)=\langle DG,-DL^{-1} \Gamma_{j-1}(G)\rangle_\mathfrak{H}.$$

\noindent
Here $D$ is the Malliavin derivative, $L^{-1}$ is the pseudo-inverse of the infinitesimal generator of the Ornstein-Uhlenbeck semigroup, and $\mathfrak{H}$ is a real separable Hilbert space. Lastly, for $f\in\mathfrak{H}^{\odot 2 }$, we write $I_{2}(f)$ for the double Wiener-It$\hat{\text{o}}$ integral of $f$.

\begin{thm}
	Let $G\in \mathbb{D}^{2,4}$ be such that $\mathbb{E}(G)=0$ and let $X\sim\text{VGD}_{1}(0,\alpha,\lambda,\lambda)$. Then, for $\lambda>1$
	
	\begin{align}
		d_{W_{3}}(G,X) \leq \frac{\lambda^{2}}{3(\lambda^{2}-1)}\mathbb{E}\left| \frac{1}{\lambda^{2}}G-\Gamma_{3}(G)  \right|+\frac{\lambda^{2}}{2(\lambda^{2}-1)}\left|\frac{2\alpha}{\lambda^{2}}-\mathbb{E}\left(\Gamma_{2}(G) \right)  \right| \label{PP2:smt1}
	\end{align}
\end{thm}

\begin{proof}
	Recall first that $\text{VGD}_{1}(0,\alpha,\lambda,\lambda) \overset{d}{=}\text{TSD}(\alpha,0,\lambda;\alpha,0,\lambda)$.

\noindent
Thus, using \eqref{PP2:e15}, we have
\begin{align}
	\nonumber	\mathbb{E}\left(h(G)-h(X)  \right)&=\mathbb{E}\left(-Gf_{h}(G)+\displaystyle\int_{\mathbb{R}}f_{h}(G+u)\nu(du)  \right)\\
	\nonumber	&=\mathbb{E}\left(-Gf_{h}(G)+ \alpha\displaystyle\int_{0}^{\infty}\left(f_{h}(G+u)-f_{h}(G-u) \right)e^{-\lambda u}du \right)\\
	\nonumber	&=\mathbb{E}\left(-Gf_{h}(G)+\frac{\alpha}{\lambda}\displaystyle\int_{0}^{\infty}\left(f_{h}^{\prime}(G+u)+f_{h}^{\prime}(G-u) \right) e^{-\lambda u}du \right)\\
	\nonumber	&=\mathbb{E}\left(-Gf_{h}(G)+\frac{2\alpha}{\lambda^{2}}f_{h}^{\prime}(G)+ \frac{\alpha}{\lambda^{2}}\displaystyle\int_{0}^{\infty}\left(f_{h}^{\prime\prime}(G+u)-f_{h}^{\prime\prime}(G-u) \right) e^{-\lambda u}du \right)\\
	\nonumber	&=\mathbb{E}\left(-Gf_{h}(G)+\frac{2\alpha}{\lambda^{2}}f_{h}^{\prime}(G)+\frac{1}{\lambda^{2}}Gf_{h}^{\prime\prime}(G)\right)\\
	\nonumber&  + \frac{1}{\lambda^{2}}\mathbb{E}\left(-Gf^{\prime\prime}_{h}(G)+\alpha\displaystyle\int_{0}^{\infty}\left(f_{h}^{\prime\prime}(G+u)-f_{h}^{\prime\prime}(G-u) \right) e^{-\lambda u}du\right)\\
	\nonumber	&=\mathbb{E}\left(-Gf_{h}(G)+\frac{2\alpha}{\lambda^{2}}f_{h}^{\prime}(G)+\frac{1}{\lambda^{2}}Gf_{h}^{\prime\prime}(G)\right)\\&  + \frac{1}{\lambda^{2}}\mathbb{E}\left(-Gf^{\prime\prime}_{h}(G)+\displaystyle\int_{\mathbb{R}}f_{h}^{\prime\prime}(G+u)\nu(du)\right)\label{PP2:a34}
\end{align} 

\noindent
Taking $\sup_{h\in\mathcal{H}_{3}}$ on both side of \eqref{PP2:a34} and rearranging the terms, we have

\begin{align}\label{PP2:a35}
	\left(1-\frac{1}{\lambda^{2}} \right)d_{W_{3}}(G,X) \leq \sup_{h\in\mathcal{H}_{3}}\left|\mathbb{E}\left(-Gf_{h}(G)+\frac{2}{\lambda^{2}}f_{h}^{\prime}(G)+\frac{1}{\lambda^{2}}f_{h}^{\prime\prime}(G)   \right)  \right|.
\end{align}

\noindent
Let $f:\mathbb{R}\to\mathbb{R}$ be a twice differentiable function with bounded first and second derivative. Then, it was shown in the proof of \cite[Theorem 4.1]{key2} that

\begin{align}
\nonumber	&\left|\mathbb{E}  \left( -Gf(G)+\frac{2}{\lambda^{2}}f^{\prime}(G)+\frac{1}{\lambda^{2}}f^{\prime\prime}(G)  \right)\right|\\
\nonumber	&=\left|\mathbb{E}\left( f^{\prime\prime}(G) \left( \frac{1}{\lambda^{2}}-\Gamma_{3}(G)  \right)+f^{\prime}(G)\left( \frac{2\alpha}{\lambda^{2}}-\mathbb{E} \left(\Gamma_{2}(G) \right)  \right)   \right)  \right|\\
	&\leq\|f^{\prime\prime}\|\mathbb{E}\left|\frac{1}{\lambda^{2}}-\Gamma_{3}(G) \right|+\|f^{\prime}\|\mathbb{E}\left|\frac{2\alpha}{\lambda^{2}}-\mathbb{E} \left(\Gamma_{2}(G)\right)  \right|\label{PP2:a36}
\end{align}

\noindent
Using \eqref{PP2:a36} in \eqref{PP2:a35}, and rearranging the terms, we have

\begin{align}
d_{W_{3}}(G,X) 	\leq \frac{\lambda^{2}}{\lambda^{2}-1}\left( \|f_{h}^{\prime\prime}\|\mathbb{E}\left|\frac{1}{\lambda^{2}}-\Gamma_{3}(G) \right|+\|f_{h}^{\prime}\|\mathbb{E}\left|\frac{2\alpha}{\lambda^{2}}-\mathbb{E} \left(\Gamma_{2}(G)\right)  \right|\right). \label{PP2:a37}
\end{align} 

\noindent
Assume that $\|h^{(k)}\|\leq 1$, for $k=2,3$. Then, by using the estimates \eqref{PP2:pr1} in \eqref{PP2:a37}, we get our desired result.
\end{proof}
\begin{rem}
	Eichelsbacher and Th$\ddot{\text{a}}$le \cite[Theorem 4.1]{key2} provide an upper bound in the Wasserstein distance for variance gamma approximation of general functionals of an isonormal Gaussian process in terms of two constants. In \cite{key2}, authors did not mention about these constants explicitly. Recently, Gaunt \cite{kk2,gauntnew} establish an upper bound in the Wasserstein distance for variance gamma approximation of general functionals of an isonormal Gaussian process, and obtain explicit constants in the main result of the article \cite[Theorem 4.1]{key2}. Note that, our bound in the Wasserstein-type distance also include the explicit constants for the symmetric variance gamma approximation.
\end{rem}

\noindent
The following corollary immediately follows for the symmetric variance-gamma approximation of double Wiener-It$\hat{\text{o}}$ integrals, that leads to the six moment theorem.

\begin{cor}
	Let $G_n=I_{2}(f_{n})$ with $f_{n}\in \mathfrak{H}^{\odot 2}$, $n\geq1$ and $X\sim\text{VGD}_{1}(0,\alpha,\lambda,\lambda)$. Then, for $\lambda>1$
	\begin{align}
	\nonumber	d_{W_{3}}(G_{n}, X)& \leq \frac{\lambda^{2}}{3(\lambda^{2}-1)}\left( \frac{1}{120}\kappa_{6}(G_n)-\frac{1}{3\lambda^{2}}\kappa_{4}(G_{n})+\frac{1}{4}\left(\kappa_{3}(G_n)\right)^{2}+\frac{1}{\lambda^{4}}\kappa_{2}(G_{n})  \right)^{\frac{1}{2}}\\
		&+\frac{\lambda^{2}}{2(\lambda^{2}-1)}\left|\frac{2\alpha}{\lambda^{2}}-\kappa_{2}(G_{n})   \right|.\label{PP2:smt2}
	\end{align}	
\end{cor}

\begin{proof}
	It is shown in \cite[Lemma 4.2 and Theorem 4.3]{nourdin0} that 
	\begin{align}\label{PP2:a38}
		\mathbb{E}\left( \Gamma_{2}(G_{n}) \right)=\kappa_{2}(G_{n}).
	\end{align}
	\noindent
	 It is also justified in \cite[Theorem 4.1]{key2} that 
	\begin{align}\label{PP2:a39}
		\mathbb{E}\left| \frac{1}{\lambda^{2}}G_{n} - \Gamma_{3}(G_{n}) \right| \leq \left(  \mathbb{E} \left(  \frac{1}{\lambda^{2}}G_{n} - \Gamma_{3}(G_{n}) \right)^{2} \right)^{\frac{1}{2}}.
	\end{align}

\noindent
It is also shown in the proof of \cite[Theorem 5.8]{key2} that
\begin{align}\label{PP2:a40}
	\mathbb{E} \left(  \frac{1}{\lambda^{2}}G_{n} - \Gamma_{3}(G_{n}) \right)^{2}= \frac{1}{120}\kappa_{6}(G_n)-\frac{1}{3\lambda^{2}}\kappa_{4}(G_{n})+\frac{1}{4}\left(\kappa_{3}(G_n)\right)^{2}+\frac{1}{\lambda^{4}}\kappa_{2}(G_{n}).
\end{align}

\noindent
Using \eqref{PP2:a40} in \eqref{PP2:a39}, we get
\begin{align}\label{PP2:a41}
\mathbb{E}\left| \frac{1}{\lambda^{2}}G_{n} - \Gamma_{3}(G_{n}) \right| \leq \left( \frac{1}{120}\kappa_{6}(G_n)-\frac{1}{3\lambda^{2}}\kappa_{4}(G_{n})+\frac{1}{4}\left(\kappa_{3}(G_n)\right)^{2}+\frac{1}{\lambda^{4}}\kappa_{2}(G_{n})  \right)^{\frac{1}{2}}.
\end{align}

\noindent
Using \eqref{PP2:a41} and \eqref{PP2:a38} in the RHS of \eqref{PP2:smt1}, we get \eqref{PP2:smt2}, as desired.
\end{proof}

\begin{rem}
Eichelsbacher and Th$\ddot{\text{a}}$le \cite[Corollary 5.10]{key2} provide an upper bound in the Wasserstein distance for variance gamma approximation of double Wiener-It$\hat{\text{o}}$ integrals in terms of two constants, and first six cumulants. In \cite{key2}, authors did not mention about these constants explicitly. Recently, Gaunt \cite{kk2,gauntnew} establish an upper bound in the Wasserstein distance for variance gamma approximation of double Wiener-It$\hat{\text{o}}$ integrals, and obtain explicit constants in the main result of the article \cite[Corollary 5.10]{key2}. Note that, our bound in the Wasserstein-type distance also include the explicit constants, and the first six cumulants for the symmetric variance gamma approximation.	
\end{rem}



\appendix
\setcounter{equation}{0}
\numberwithin{equation}{section}

\section{}\label{app}

\noindent
In this section, we derive some properties of TSD, and prove some results that are used in the previous sections. We also prove the results that we stated in the main text without proofs.

\subsection{Self-decomposability}
Here, we discuss self-decomposable property of TSD that are used to estimate the properties of the solution to Stein equation. 
\begin{lem}\label{p3}
	Let $X\sim \text{TSD}(\alpha^{+},\beta^{+},\lambda^{+};\alpha^{-},\beta^{-},\lambda^{-})$. Then, $X$ has the self-decomposable property.
\end{lem}
\begin{proof}
	In view of L\'evy measure (\ref{e2}), we can express the cf of TSD (\ref{e1}) as
	
	\begin{equation}
		\phi(t)=\exp\left\{\int_{\mathbb R}(e^{itu}-1)\frac{k(u)}{u}  du \right \},~~t\in\mathbb{R},
	\end{equation}
	
	\noindent
	where $k:\mathbb{R}\to\mathbb{R}$ denotes the function
	
	\begin{equation}
		k(u)=\frac{\alpha^{+}  }{u^{\beta^{+}}}e^{-\lambda^{+}u}\mathbf{1}_{(0,\infty)}(u)-\frac{\alpha^{-}  }{|u|^{\beta^{-}}}e^{-\lambda^{-}|u|}\mathbf{1}_{(-\infty,0)}(u), ~~u\in\mathbb{R}.
	\end{equation}
	
	\noindent
	Note that $k\geq0$ on $(0,\infty)$ and $k\leq0$ on $(-\infty,0)$. Again, $k$ is strictly decreasing on $(-\infty,0)$ and $(0,\infty)$. It is an immediate consequence of (Sato \cite[Corollary 15.11]{sato}) that TSD are self-decomposable.
	
\end{proof}

\subsection{Frullani improper integral}
\noindent
Here we discuss a special type of improper integral, so called Frullani integral that is used to prove similarity of various cf representations of VGD. We prove a lemma on this integral

\begin{lem}
	Let a fuction $g:(0,\infty)\to \mathbb{R}$ is differentiable on $(0,\infty)$. Let $\lim_{x\to 0^{+}}g(x)$ and $\lim_{x\to \infty}g(x)$ exist finitely, and the limiting values are $g_0$ and $g_\infty$ respectively.Then
	
	\begin{equation*}
	\int_{0}^{\infty}\frac{g(ax)-g(bx)}{x}dx=(g_{0}-g_{\infty})\log\left(\frac{b}{a} \right),
	\end{equation*}
	\noindent
	where $b,a>0$.
\end{lem}
\begin{proof}
We prove this lemma by extending the integrand, and using Fubini's theorem. Note that
\begin{align*}
	\displaystyle\int_{0}^{\infty}\frac{g(ax)-g(bx)}{x}dx&=\displaystyle\int_{0}^{\infty}\int_{b}^{a}\frac{g^{\prime}(xt)}{x}dtdx\\
	&=\int_{b}^{a}\int_{0}^{\infty}\frac{g^{\prime}(xt)}{x}dtdx\\
	&=\int_{b}^{a}(g_\infty-g_0)\frac{1}{x}dx\\
	&=(g_0-g_\infty)\log\left( \frac{b}{a} \right).
\end{align*}
\noindent
Hence the lemma is proved.

\space

\noindent
As for example, consider the integral $I=\displaystyle\int_{0}^{\infty}\frac{e^{-\lambda_{1}x}-e^{-\lambda_{2}x}}{x}dx,$ where $\lambda_{1},\lambda_{2}>0$. Then by the above lemma one can easily obtain $I=\log\left(\frac{\lambda_{2}}{\lambda_{1}}\right)$.
	
\end{proof}
\noindent
Next, we prove a technical result used in the previous section.

\begin{pro}\label{PP2:appendixPro1}
	Let $x,z \in \mathbb{R}$. Then, for all $t\geq 0,$
	\begin{equation}\label{PP2:e48}
		\lim_{t\to0^{+}}\frac{1}{t}\left(e^{iz x(e^{-t}-1)}\phi_{t}(z)-1\right)=\left(-x+ \displaystyle\int_{\mathbb{R}}ue^{iz u}\nu(du) \right )(iz).
	\end{equation}
\end{pro}

\begin{proof}
	Recall from Section \ref{mr}, if $X$ be a tempered stable random variable, we write
	\begin{align*}
		\phi_{t}(z)=\frac{\phi(z)}{\phi(e^{-t}z)}=\exp\left(\displaystyle\int_{\mathbb{R}}(e^{iz u}-e^{iue^{-t}z})\nu(du)\right),~~t\geq 0\text{ (see \eqref{PP2:a15})}.
	\end{align*}

	\noindent
	Now, let us consider LHS of \eqref{PP2:e48},
	\begin{align}\label{PP2:smy1}
		\nonumber &	\lim_{t\to0^{+}}\frac{1}{t}\left(e^{iz x(e^{-t}-1)}\phi_{t}(z)-1\right)\\
		\nonumber	=&\lim_{t\to0^{+}}\frac{1}{t}\left(\exp\left(izx(e^{-t}-1)+ \displaystyle\int_{\mathbb{R}}(e^{iz u}-e^{iue^{-t}z })\nu(du) \right) -1     \right)\\
		=&\lim_{t\to0^{+}}\frac{1}{t}\left(\exp \left(A+iB\right)-1    \right),
	\end{align}
	\noindent
	where 
	\begin{align*}
		A&=\int_{\mathbb{R}}(\cos(zu) -\cos(zue^{-t}))\nu_\alpha(du) \text{ and}\\
		B&=\left(zx(e^{-t}-1)+\int_{\mathbb{R}}(\sin(zu)- \sin(zue^{-t}) ) \nu(du)  \right).
	\end{align*}

	\noindent
	Applying Euler's formula for complex exponential to \eqref{PP2:smy1}, and rearranging the limits, we have
	\begin{align}\label{PP2:smy2}
		\lim_{t\to0^{+}}\frac{1}{t}\left(e^{iz x(e^{-t}-1)}\phi_{t}(z)-1\right)&=\lim_{t\to0^{+}}\frac{e^{A}\cos(B)-1}{t}+i\lim_{t\to0^{+}}\frac{e^{A}\sin(B)}{t}.
	\end{align}
	
	\noindent
	It is easy to show that at $t=0$, $e^{A}\cos(B)-1=0$ and $e^{A}\sin(B)=0.$ Thus, on applying L'Hospital rule on \eqref{PP2:smy2}, taking limit as $t$ tend to $0^{+}$, and using dominated convergence theorem, we have
	
	\begin{align}
		\nonumber \lim_{t\to0^{+}}\frac{1}{t}\left(e^{iz x(e^{-t}-1)}\phi_{t}(z)-1\right)&=\left(\int_{\mathbb{R}}iu\sin(zu)\nu(du) -x+\int_{\mathbb{R}}u\cos(zu)\nu(du)  \right)(iz)\\
		\nonumber &=\left(-x+\int_{\mathbb{R}}u(\cos(zu)+i\sin(zu) )\nu(du)   \right)(iz)\\
		\nonumber &=\left(-x+\int_{\mathbb{R}}ue^{izu}\nu(du)  \right)(iz)
	\end{align}
	
	\noindent
	This completes the proof.
\end{proof}

\subsection{Further proofs} 

\subsubsection{Proof of the converse of Theorem \ref{th1}} For any $s \in \mathbb{R}$, let $f(x)=e^{isx}$, $x\in \mathbb{R}$, then (\ref{PP2:StenIdTSD}) becomes
\begin{align*}
\mathbb{E}Xe^{isX} &= \mathbb{E}\int_{\mathbb R}e^{is(X+u)}u\nu(du)\\
&=\mathbb{E}e^{isX}\int_{\mathbb R}e^{isu}u\nu(du).
\end{align*}

\noindent
Setting $\phi(s)=\mathbb{E}e^{isX}$, Then
\begin{equation}\label{suff1}
\phi^{\prime}(s)=i\phi(s)\int_{\mathbb R}e^{isu}u\nu(du).
\end{equation}
\noindent
Integrating out the real and imaginary parts of (\ref{suff1}) leads, for any $t\geq 0$, to

\begin{align*}
\phi(t)&=\exp \left(i\int_{0}^{t}\int_{\mathbb R}e^{isu}u\nu(du)ds  \right)\\
&=\exp\left(i\int_{\mathbb R}\int_{0}^{t}e^{isu}dsu\nu(du)  \right)\\
&=\exp\left(\int_{\mathbb R}(e^{itu}-1)\nu(du)  \right).
\end{align*}

\noindent
A similar computation can be done for $t\leq0.$ Hence the converse part of the theorem is proved.
\subsubsection{Proof of Proposition \ref{PP2:proSem}}  For each $f\in\mathcal{F}_X$, it is easy to show that $P_{0}f(x)=f(x) ~~\text{and}~~ \lim_{t\to\infty}P_{t}(f)(x)=\displaystyle\int_{\mathbb R}f(x)F_{X}(dx)$. 
\noindent
Now, for any $s,t\geq0$, we have
\begin{align}\label{PP2:a22}
	\phi_{t+s}(z)=\frac{\phi(z)}{\phi(e^{-(t+s)}z)}
	=\frac{\phi(z)}{\phi(e^{-s}z)}\frac{\phi(e^{-s}z)}{\phi(e^{-(t+s)}z)}
	=\phi_{s}(z)\phi_{t}(e^{-s}z)
\end{align}
\noindent
Using \eqref{PP2:a22}, we have
\begin{align}
	\nonumber	LHS=P_{t+s}(f)(x)&=\frac{1}{2\pi}\int_{\mathbb R}\widehat{f}(z)e^{iz xe^{-(t+s)}}\phi_{t+s}(z)dz\\
	\label{PP2:a022}	&=\frac{1}{2\pi}\int_{\mathbb R}\widehat{f}(z)e^{iz xe^{-(t+s)}}\phi_{s}(z)\phi_{t}(e^{-s}z) dz.
\end{align}
\noindent
We need to show that $P_{t+s}(f)(x)=P_{t}(P_{s}f)(x)$ for all $f\in \mathcal{F}_X$. Let $\delta$ be the Dirac-$\delta$ measure.

\begin{align*}
	\text{RHS}&=P_{t}(P_{s}(f))(x)\\
	&=\frac{1}{2\pi}\int_{\mathbb R}\widehat{P_{s}(f)}(z)e^{iz xe^{-t}}\phi_{t}(z)dz\\
	&=\frac{1}{2\pi}\int_{\mathbb R}\left(\int_{\mathbb{R}}e^{-ivz}P_{s}(f)(v)dv  \right)e^{iz xe^{-t}}\phi_{t}(z)dz\\
	&=\frac{1}{(2\pi)^{2}}\int_{\mathbb R}\left(\int_{\mathbb{R}}e^{-ivz}\left(\int_{\mathbb{R}}\widehat{f}(w)e^{iwe^{-s}v}\phi_{s}(w)dw  \right)dv  \right)e^{iz xe^{-t}}\phi_{t}(z)dz\\
	&=\frac{1}{(2\pi)^{2}}\int_{\mathbb R}\widehat{f}(w)\phi_{s}(w)\int_{\mathbb R}e^{iz xe^{-t}}\phi_{t}(z) \left(\int_{\mathbb R}e^{iv(e^{-s}w-z)}dv  \right)dz dw\\
	&=\frac{1}{(2\pi)^{2}}\int_{\mathbb R}\widehat{f}(w)\phi_{s}(w)\int_{\mathbb R}e^{iz xe^{-t}}\phi_{t}(z) 2\pi \delta (e^{-s}w -z) dz dw\\
	&=\frac{1}{2\pi}\int_{\mathbb R}\widehat{f}(w)\phi_{s}(w)e^{i e^{-s}wxe^{-t}}\phi_{t}(e^{-s}w)dw\\
	&=\frac{1}{2\pi}\int_{\mathbb R}\widehat{f}(z)e^{iz xe^{-(t+s)}}\phi_{s}(z)\phi_{t}(e^{-s}z) dz\\
	&=P_{t+s}(f)(x)=\text{LHS}~~(\text{from }\eqref{PP2:a022}),
\end{align*}   
\noindent
and the desired conclusion follows.

\subsubsection{Remaining proof of Theorem \ref{thmsol}} Let us consider a function $g_{h}: \mathbb{R}\to \mathbb{R}$ defined as
\begin{equation*}
	g_{h}(x)=-\displaystyle\int_{0}^{\infty}\left(P_{t}(h)(x)-\mathbb{E}h(X)  \right)dt,~~h\in \mathcal{H}_{r},
\end{equation*}
\noindent
where $(P_{t})_{t\geq0}$ is the semigroup defined in \eqref{PP2:e16}. 

%

\noindent
Using \eqref{PP2:a17}, we have
\begin{align}
	\nonumber |P_{t}(h)(x)-\mathbb{E}h(X)|&=\left|\int_{\mathbb R}h(y+e^{-t}x)F_{X_{(t)}}(dy)-\int_{\mathbb R}h(y)F_{X}(dy) \right|\\
	\nonumber&=\left|\int_{\mathbb R}(h(y+e^{-t}x)-h(y))F_{X_{(t)}}(dy)\right.\\
	\nonumber&\left.+\int_{\mathbb{R}}h(y)F_{X_{(t)}}(dy) -\int_{\mathbb R}h(y)F_{X}(dy)  \right|\\
	\nonumber &\leq e^{-t}|x| |h^{(1)}|+\left|\int_{\mathbb{R}}\widehat{h}(z)\left(\phi_{t}(z)-\phi(z)  \right)dz\right|\\
	&\leq e^{-t}|x| |h^{(1)}|+\int_{\mathbb{R}}|\widehat{h}(z)||\phi_{t}(z)-\phi(z) |dz  \label{PP2:ell00}
\end{align}
\noindent
Now, let us calculate an upper bound between the difference of two characteristic functions $\phi_{t}$ and $\phi$. For all $t>0$ and $z\in \mathbb{R},$
\begin{align*}
	|\phi_{t}(z)-\phi(z)|=\left|\frac{\phi(z)}{\phi(e^{-t}z)}-\phi(z)\right|\leq \left|\phi_{\alpha}(e^{-t}z)-1   \right|=|e^{\omega_{t}(z)}-1|,
\end{align*}
\noindent
where $\omega_{t}(z)=\int_{\mathbb{R}}(e^{ize^{-t}u}-1)\nu(du)$. Note that the function $z\to e^{s\omega_{t}(z)}$ is a characteristic function for all $s\in (0,\infty)$. Thus, for all $z\in \mathbb{R}$ and $t>0$,
\begin{align}
	\nonumber |\phi_{t}(z)-\phi(z)|&\leq \left| \int_{0}^{1}\frac{d}{ds}(\exp (s\omega_{t}(z))) ds  \right|\\
	\nonumber &\leq |\omega_{t}(z)|\\
	\nonumber &\leq \max\{e^{-t\beta^{+}},e^{-t\beta^{-}}\}\left|\int_{\mathbb{R}}(e^{izu}-1)\tilde{\nu}(du)  \right|\\
	& \leq \max\{e^{-t\beta^{+}},e^{-t\beta^{-}}\}C\left(1+|z|^{2}\right),~~C>0, \label{PP2:ell01}
\end{align}
\noindent
where $
\tilde{\nu}(du)=\left(\frac{\alpha^{+}  }{u^{1+\beta^{+}}}\mathbf{1}_{(0,\infty)}(u)+\frac{\alpha^{-}  }{|u|^{1+\beta^{-}}}\mathbf{1}_{(-\infty,0)}(u)\right)du$, and the last inequality is followed by \cite[p.30, Ex. 1.2.16]{newref1}. Using \eqref{PP2:ell01} in \eqref{PP2:ell00}, one can easily show that $\displaystyle\int_{0}^{\infty}|P_t(h)(x)-\mathbb{E}h(X)|dt<\infty$. Hence, $g_{h}(x)$ is well-defined.  \noindent
By dominated convergence theorem, we see that $g_{h}$ is differentiable and
\begin{align}
	\nonumber g_{h}^{\prime}(x)&=-\lim_{\zeta\to\infty} \frac{d}{dx}\displaystyle\int_{0}^{\zeta}(P_t(h)(x)-\mathbb{E}h(X))dt\\
	\nonumber&=-\lim_{\zeta\to\infty}\displaystyle\int_{0}^{\zeta}\frac{d}{dx}\left(\int_{\mathbb{R}}h(xe^{-t}+u)F_{X_{(t)}}(du)  \right)dt\\
	\nonumber&=-\displaystyle\int_{0}^{\infty}e^{-t}\int_{\mathbb{R}}h^{\prime}(u+xe^{-t})F_{X_{(t)}}(du)dt=f_{h}(x),
\end{align}

\noindent
the desired conclusion follows.

\noindent 
\textbf{Acknowledgment}: The first author is thankful for the financial support
of HTRA fellowship at IIT Madras. The second author is thankful for the funding by IIT Madras from the IoE project: SB20210848MAMHRD008558.


\begin{thebibliography}{}\small

	\bibitem{newref1} Applebaum, D. (2009). L\'evy processes and stochastic calculus, Second edition. Cambridge Studies in Advanced Mathematics, $\bf{116}$. Cambridge University Press, Cambridge, xxx+460 pp.

\bibitem{frullani} Arias-De-Reyna, J. (1990). On the Theorem of Frullani. Proceedings of the American Mathematical Society. Volume $\bf{109}$. Number I, pp. 165-175.


\bibitem{k0} Arras, B. and Houdr$\acute{e}$, C.(2019). On Stein's method for infinitely divisible laws with finite first moment. Springer Briefs in Probability and Mathematical Statistics, Springer.


\bibitem{k22}Arras, B. and Houdr\'e, C. (2019). On Stein’s method for multivariate self-decomposable laws. Preprint
arXiv:1907.10050.


\bibitem{k8} Barbour, A. D. (1990). Stein's method for diffusion approximations. Probability Theory and Related Fields $\mathbf{84}$, pp. 297-322.


\bibitem{k14} Boyarchenko, S.I. and Levendorskii,  S.Z. (2000). Option pricing for truncated L\'evy processes, International Journal of Theoretical and Applied Finance $\mathbf{3(3)}$, pp. 549-552.


\bibitem{k15}Carr, P., Geman, H., Madan,  D.B., Yor, M. (2002). The fine structure of asset returns: an empirical investigation. Journal of Business $\mathbf{75 (2)}$ pp. 305-332.



\bibitem{k3}Chen, L. H. Y.(1975). Poisson approximation for dependent trials. Annals of Probability $\mathbf{3}$, pp. 534-545.

\bibitem{chen}Chen, L. H. Y., Goldstein, L. and Shao, Q-M. (2011). Normal Approximation by Stein's Method. Springer. 

\bibitem{k16}Chen, P., Nourdin, I., Xu, L., Yang, X., Zhang, R. (2019). Non-integrable stable approximation by Stein’s method. Preprint http://arxiv.org/abs/1903.12315v2


\bibitem{k17}Chen, P., Nourdin, I. and Xu, L. (2018). Stein’s method for asymmetric $\alpha-$stable distributions, with applications to CLT. Preprint https://arxiv.org/pdf/1808.02405.


\bibitem{kk6} D$\ddot{o}$bler, C., Gaunt, R.E. and Vollmer, S. J. (2017). An iterative technique for bounding derivatives of solutions of Stein equations. Electron. J. Probab. 22, no. 96, pp. 1-39


\bibitem{k18}Eichelsbacher, P. and Reinert, G. (2008). Stein’s method for discrete Gibbs measures. The Annals of Applied Probability, 18, pp. 1588-1618.

\bibitem{key2}Eichelsbacher, P. and Th$\ddot{a}$le, C. (2015). Malliavin-Stein method for variance-gamma approximation on Wiener space. Electron. J. Probab. 20, no. $\mathbf{123}$ , pp. 1-28.



\bibitem{kk1}Finlay, R. and Seneta, E. (2008). Option pricing with VG-like models. International Journal of Theoretical and Applied Finance
Vol. 11, No. $\mathbf{8}$, pp. 943-955


\bibitem{newref2} Fofack, H., Nolan, J.P. (1999). Tail Behavior, Modes and other Characteristics of Stable Distributions. Extremes 2, 39–58 . https://doi.org/10.1023/A:1009908026279 


\bibitem{k4}Fulman, J. and Ross, N. (2013). Exponential approximation and Stein's method of exchangeable pairs. ALEA, Latin American Journal of Probability and Mathematical Statistics $\mathbf{10(1)}$, pp. 1-13.



\bibitem{N2}Gaunt, R. E. (2013). Rates of convergence of variance-gamma approximations via Stein's method. Dphil thesis, The Queen's College
University of Oxford.

\bibitem{k24}Gaunt, R. E. (2014). Variance-Gamma approximation via Stein’s method. Electronic Journal of Probability 19 no. 38, pp. 1-33.

\bibitem{kk2}Gaunt, R.E. (2020). Wasserstein and Kolmogorov error bounds for
variance-gamma approximation via Stein’s method $\textbf{I}$. Journal of Theoretical Probability $\mathbf{33}$, pp. 465-505

\bibitem{gauntnew}Gaunt, R.E. (2020). Stein factors for variance-gamma approximation in the Wasserstein and Kolmogorov distances. preprint: https://arxiv.org/pdf/2008.06088


\bibitem{k6}Gaunt, R.E., Pickett, A.M., Reinert, G. (2017). Chi-square approximation by Stein's method with application to Pearson's statistic. Annals of Applied Probability $\mathbf{27 (2)}$, 720-756

\bibitem{kk3}Gaunt, R.E. (2017). On Stein's method for products of normal random variables and zero bias couplings. Bernoulli $\mathbf{23(4B)}$, pp. 3311-3345 

\bibitem{key1}Gaunt, R.E., Mijoule, G. and Swan, Y. (2019). Some new Stein operators for product distributions. Preprint https://arxiv.org/abs/1901.11460v2

\bibitem{gaunt7} Gaunt, R.E. (2021). New error bounds for Laplace approximation via Stein's method. ESAIM: Probability and Statistics, \textbf{25}, pp. 325-345.

\bibitem{kk5}Grabchak, M. (2015). Tempered stable distributions stochastic models for multiscale processes. Springer Briefs in Mathematics, Springer.


\bibitem{koponen} Koponen, I. (1995). Analytic approach to the problem of convergence of truncated L\'evy flights towards the Gaussian stochastic process. Physical Review E $\mathbf{52}$. pp. 1197-1199

\bibitem{k13} K$\ddot{u}$chler, U. and Tappe, S (2013). Tempered stable distributions and processes. Stochastic Stochastic Processes and their Applications $\mathbf{123}$, 4256-4293



\bibitem{k19}Kumar, A.N. and Upadhye, N.S. (2017). On discrete Gibbs measure approximation to runs. Preprint https://arxiv.org/pdf/1701.03294


\bibitem{ley}Ley, C., Reinert, G., Swan, Y. (2017). Stein's method for comparison of univariate distributions. Probab. Surv. \textbf{14}, pp. 1-52. 



\bibitem{k5} Luk, H. (1994). Stein's Method for the Gamma Distribution and Related Statistical Applications. PhD thesis, University of Southern California.

\bibitem{nourdin} Nourdin, I. and Peccati, G. (2012). Normal approximations with Malliavin calculus. Cambridge University Press. Cambridge tracts in mathematics \textbf{192}.

\bibitem{nourdin0} Nourdin, I. and Peccati, G. (2010). Cumulants on the Wiener space. J. Funct. Anal. 258, pp. 3775-3791.

\bibitem{pike}Pike, J. and Ren, H. (2014). Stein's method and the Laplace distribution. ALEA Lat. Am. J. Probab. Math. Stat. $\bf{11}$, pp. 571–587.

\bibitem{k9}Rosinski, J. (2007). Tempering stable processes. Stochastic Processes and their Applications $\mathbf{117(6)}$, pp. 677–707.

\bibitem{k25}Ross, N. (2010). Fundamentals of Stein’s method. Probability Surveys $\mathbf{8}$, pp. 210-293.

\bibitem{sato} Sato, K.I. (1999). L\'evy Processes and Infinitely Divisible Distributions. Cambridge University Press, Cambridge.

\bibitem{slepov} Slepov, N.A. (2021). Convergence rate of random geometric sum distributions to the Laplace law. Theory Probab. Appl., \textbf{66(1)}, 121-141.

\bibitem{k2}Stein, C. (1972). A bound for the error in the normal approximation to the the distribution of a sum of dependent random variables. In Proceedings of the Sixth Berkeley Symposium on Mathematical Statistics and Probability, $\mathbf{2}$, Univ. California Press, Berkeley, pp. 583-602.

\bibitem{den}Stein, C., Diaconis, S, Holmes, S. and Reinert, G. (2004). Use of exchangeable pairs in the analysis of simulations. In Steins method: expository lectures and applications, volume 46 of IMS Lecture Notes Monogr.Ser., Inst. Math. Statist., Beachwood, OH , pp. 1-26.

\bibitem{stein} Stein, E.M. and Shakarchi, R. (2003). Fourier analysis. An introduction. Princeton Lectures in Analysis, 1. Princeton.

\bibitem{k10}Sztonyk, P. (2010). Estimates of tempered stable densities. Journal of Theoretical Probability $\mathbf{23(1)}$, pp. 127–147.



\bibitem{k23}Schoutens, W. (2001). Orthogonal polynomials in Stein’s method. Journal of Mathematical Analysis and Applications $\mathbf{253}$, pp. 515-531.




\bibitem{k1}Upadhye, N.S. and Barman, K. (2020). A unified approach to Stein's method for stable distributions. Preprint https://arxiv.org/pdf/2004.07593






































\bibitem{k20}Xu, L. (2019). Approximation of stable law in Wasserstein-1 distance by Stein’s method. The Annals of Applied Probability $\mathbf{29}$, No. 1, pp. 458-504.





\end{thebibliography}
\end{document}